\documentclass[11pt,final]{amsart}

\usepackage[T1]{fontenc}
\usepackage[utf8]{inputenc}

\usepackage[english]{babel}
\usepackage{geometry}
\usepackage{graphicx}

\usepackage{xcolor}
\usepackage{amsmath,amssymb,amsthm,amsfonts,stmaryrd,mathtools}
\usepackage{bbm}
\usepackage{eurosym}
\usepackage{array,dcolumn,mathdots,multirow}
\usepackage{rotating}
\usepackage{fancyvrb}
\usepackage[activate={true,nocompatibility},spacing,kerning]{microtype}

\graphicspath{%
{./Images/}
}

\definecolor{darkgreen}{rgb}{0,0.5,0}
\definecolor{darkred}{rgb}{0.5,0,0}
\definecolor{darkblue}{rgb}{0.004,0.396,0.741}
\definecolor{gcolor}{rgb}{0.004,0.396,0.741}	
\definecolor{warning}{rgb}{1.0,0.6,0.0}	
\definecolor{gray}{rgb}{0.6,0.6,0.6}

\usepackage[pdftex,
            colorlinks,
            hyperfootnotes=false,
            pdffitwindow=true,
            plainpages=false,
            pdfpagelabels=true,
            pdfpagemode=UseOutlines,
            pdfpagelayout=TwoPageRight,
            pdftitle={A Stirling-type formula related to longest increasing subsequences},
            pdfauthor={Folkmar Bornemann, TU Munich},
            hyperindex,
            setpagesize=false,
            filecolor=purple,
            urlcolor=darkred,
            citecolor=gcolor,
            linkcolor=gcolor]{hyperref}
\usepackage{remreset}
\usepackage{bm}

\usepackage{caption}
\theoremstyle{plain}
\newtheorem{theorem}{Theorem}[section]
\newtheorem{lemma}{Lemma}[section]
\newtheorem{corollary}{Corollary}[section]
\theoremstyle{definition}
\theoremstyle{remark}
\newtheorem{remark}{Remark}[section]
\newtheorem*{note}{Remark}

\input{macros.sty}
\VerbatimFootnotes
\makeindex
\begin{document}

\title[A Stirling-type formula related to longest increasing subsequences]{A Stirling-type formula for the distribution of the length of longest increasing subsequences}
\author{Folkmar Bornemann}
\address{Department of Mathematics, TU München, Germany}
\email{bornemann@tum.de}


\begin{abstract}
The discrete distribution of the length of longest increasing subsequences in random permutations of $n$ integers is deeply related to random matrix theory. In a seminal work, Baik, Deift and Johansson provided an asymptotics in terms of the distribution of the scaled largest level of the large matrix limit of GUE. As a numerical approximation, however, this asymptotics is inaccurate for small $n$ and has a slow convergence rate, conjectured to be just of order $n^{-1/3}$.
Here, we suggest a different type of approximation, based on Hayman's generalization of Stirling's formula. Such a formula gives already a couple of correct digits of the length distribution for $n$ as small as $20$ but allows numerical evaluations, with a uniform error of apparent order $n^{-2/3}$, for $n$ as large as $10^{12}$; thus closing the gap between a table of exact values (compiled for up to $n=1000$) and the random matrix limit.
Being much more efficient and accurate than Monte-Carlo simulations, the Stirling-type formula allows for a precise numerical understanding of the first few finite size correction terms to the random matrix limit. From this we derive expansions of the expected value and variance of the length, exhibiting several more terms than previously put forward.
\end{abstract}

\keywords{Random permutations, random matrices, $H$-admissibility, Stirling-type formula}
\subjclass[2010]{05A16, 60B20, 30D15, 47N40}

\maketitle

\section{Introduction} 

As witnessed by a number of outstanding surveys and monographs (see, e.g., \cite{MR1694204, MR3468920, MR3468738, MR2334203}), a surprisingly rich topic in combinatorics and probability theory, deeply related to representation theory and to random matrix theory, is the study of the lengths $L_n(\sigma)$ of longest increasing subsequences of permutations $\sigma$ on the set $[n] = \{1,2,\ldots,n\}$ and of the behavior of their distribution in the limit $n\to\infty$. Here, $L_n(\sigma)$ is defined as the maximum of all $k$ for which there are $1\leq i_1 < i_2 < \cdots < i_k \leq n$ such that $\sigma_{i_1} < \sigma_{i_2} < \cdots < \sigma_{i_k}$. Writing permutations in the form $\sigma = (\sigma_1\, \sigma_2\,\cdots\, \sigma_n)$ we get, e.g.,  $L_9(\sigma)= 5$ for $\sigma = (4\,{\bf 1}\, {\bf 2}\, 7\, 6\, {\bf 5}\, {\bf 8}\,{\bf 9} \,3)$, where  one of the longest increasing subsequences has been highlighted. Enumeration of the permutations with a given $L_n$ can be encoded probabilistically: by equipping the symmetric group on~$[n]$ with the uniform distribution, $L_n$ becomes a discrete random variable with cumulative probability distribution (CDF) $\prob( L_n \leq l)$ and probability distribution (PDF) $\prob( L_n = l)$.

\subsection*{Constructive combinatorics.} Using the Robinson--Schensted correspondence \cite{MR121305}, one gets the distribution of $L_n$ in the following form (see, e.g., \cite[§§3.3--3.7]{MR843332}):
\begin{equation}\label{eq:hook}
\prob( L_n \leq l) = \frac1{n!} \sum_{\lambda \,\vdash n\,:\, l_\lambda \leq l} d_\lambda^2.
\end{equation}
Here $\lambda \,\vdash n$ denotes an integer partition $\lambda_1\geq \lambda_2 \geq \cdots \geq \lambda_{l_\lambda}> 0$ of $n=\sum_{j=1}^{l_\lambda} \lambda_j$ and $d_\lambda$ is the number of standard Young tableaux of shape $\lambda$, as given by the hook length formula. By generating all partitions $\lambda\,\vdash n$, in 1968 Baer and Brock \cite{MR228216} computed tables of $\prob( L_n = l)$ up to $n=36$; in 2000 Odlyzko and Rains \cite{MR1771285} for $n=15,30,60,90,120$ (the tables are online, see \cite{OdlyzkoTable}), reporting a computing time for $n=120$ of about 12 hours (here $p_n$, the number of partitions, is of size $1.8\times 10^9$). This quickly becomes infeasible,\footnote{See Sect.~\ref{subsect:exact} below for a method to compute the exact rational values of the distribution of $L_n$ based on random matrix theory, which has been used by the author to tabulate $\prob(L_n=l)$, $1\leq l\leq n$, up to $n=1000$.} as $p_n$ is already as large as $2.3\cdot 10^{14}$ for $n=250$. Another use of the combinatorial methods is the approximation of the distribution of $L_n$ by Monte Carlo simulations \cite{MR228216, MR1771285}: one samples random permutations~$\sigma$ and calculates $L_n(\sigma)$ by the Robinson--Schensted correspondence.\footnote{In {\em Mathematica}, a single trial is generated by the command 
\[
{\text {\tt Length@LongestOrderedSequence[PermutationList[RandomPermutation[n],n]]}}.
\]\\*[-7mm]
}

\begin{figure}[tbp]
\includegraphics[width=0.455\textwidth]{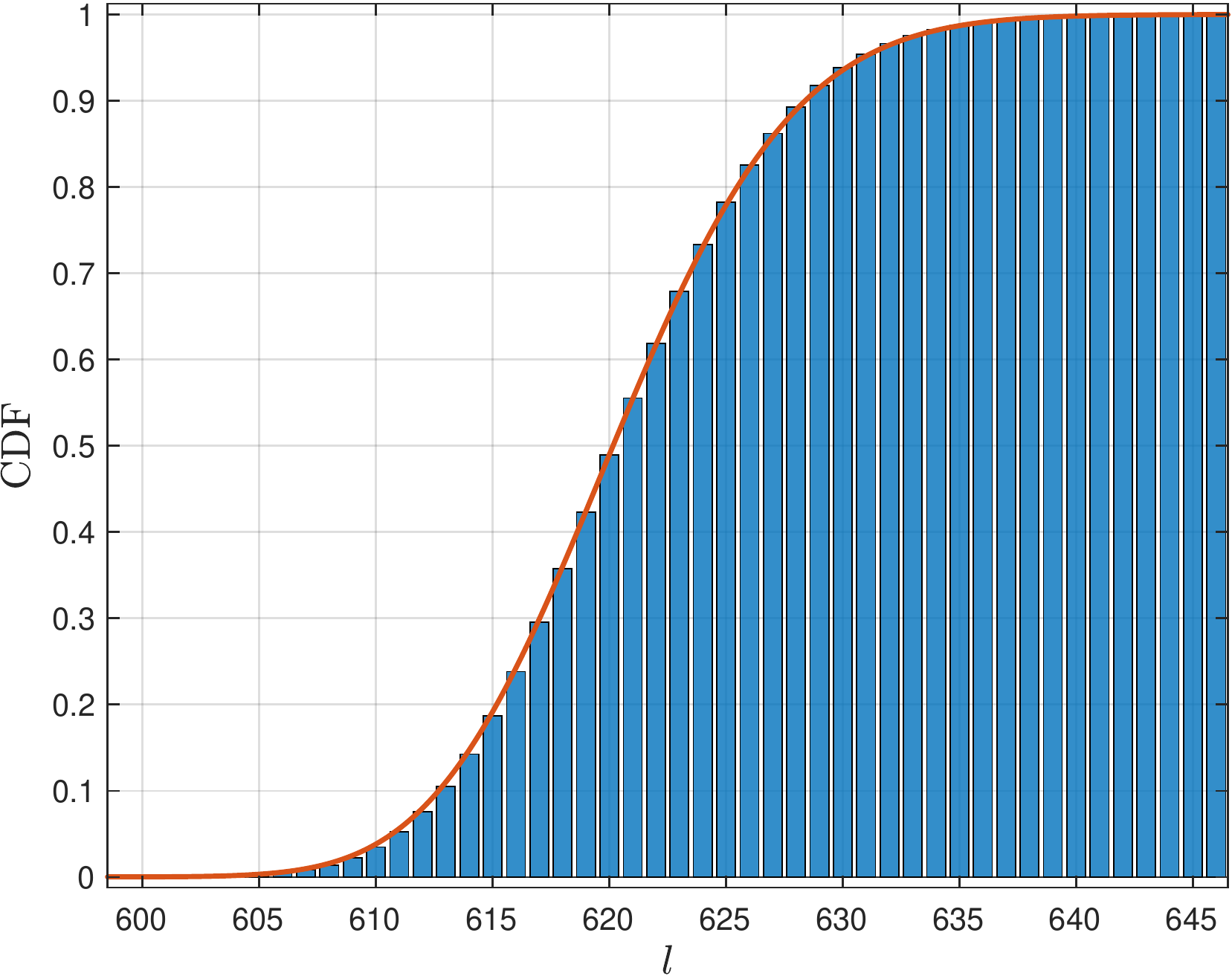}\hfil
\includegraphics[width=0.455\textwidth]{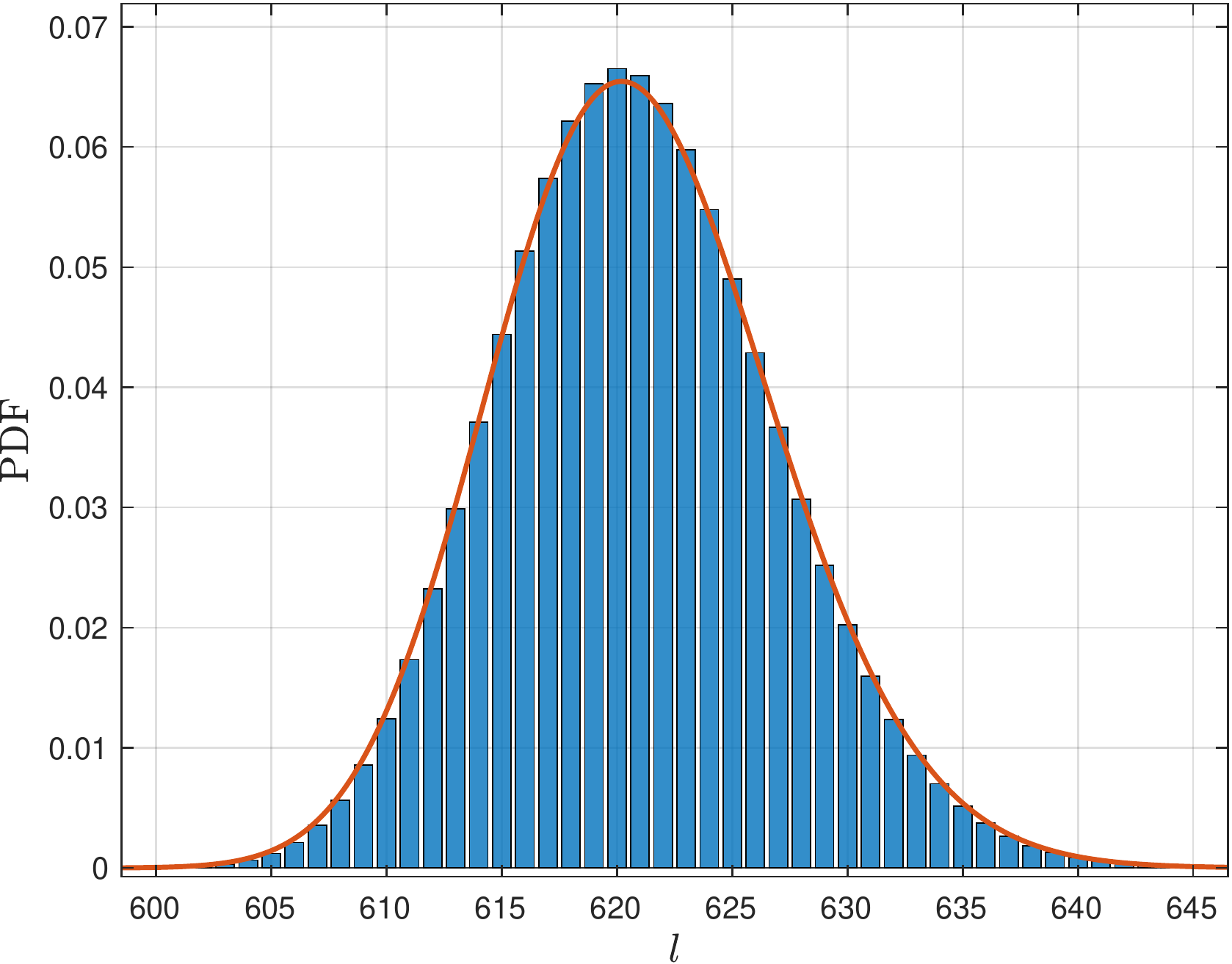}
\caption{{\footnotesize Discrete distribution of $L_n$ for $n=10^5$ near its mode vs. the random matrix limit given by the leading order terms in~\eqref{eq:BDJ99} and \eqref{eq:PDFexpansion} (solid red line); here and in the figures below, discrete distributions are shown as blue bars centered at the integers. Left: CDF $\prob( L_n \leq l)$; right: PDF $\prob( L_n = l)$. The discrete distributions were computed using the Stirling-type formula \eqref{eq:stirling}, with additive errors estimated to be smaller than $10^{-5}$, cf. Fig.~\ref{fig:error}, which is well below plotting accuracy.}}
\label{fig:pdf100000}
\end{figure}

\subsection*{Analytic combinatorics and the random matrix limit.} For analytic methods of enumeration the starting points is a more or less explicit representation of a suitable generating function; here, the suitable one turns out to be the exponential generating function of the CDF $\prob( L_n \leq l)$, when considered as a sequence of $n$ with the length $l$ fixed:
\[
f_l(z) := \sum_{n=0}^\infty \frac{\prob( L_n \leq l)}{n!}z^n \qquad (z\in\C,\, l \in \N).
\]
(We note that $f_l$ is an entire function of exponential type.) In fact,  Gessel \cite[p.~280]{Gessel90} obtained in 1990 the explicit representation 
\begin{equation}\label{eq:Gessel}
f_l(z^2) = D_l(z),\qquad D_l(z) :=\det_{j,k=1}^l I_{j-k}(2z),
\end{equation}
in terms of a Toeplitz determinant of the modified Bessel functions $I_m$, $m\in \Z$, which are entire functions of exponential type themselves. By relating, first, the Toeplitz determinant to the machinery of Riemann--Hilbert problems to study a double-scaling limit of the generating function and by using, next, a Tauberian theorem\footnote{The Tauberian part (``de-Poissonization'') makes it hard to get more than the leading order of the asymptotics.}  to induce from that limit an asymptotics of the coefficients, Baik, Deift and Johans\-son~\cite{MR1682248} succeeded 1999 in establishing\footnote{In the first place,  \cite[Thm.~1.1]{MR1682248} states the following limit to hold {\em pointwise} in $t\in \R$:
\[
\lim_{n\to\infty} \prob\left(\frac{L_n-2\sqrt{n}}{n^{1/6}}\leq t\right) = F_2(t).
\]
However, since the limit distribution $F_2$ is continuous, by a standard Tauberian follow-up \cite[Lemma 2.1]{MR1652247} of the Portmanteau theorem in probability theory, this convergence is known to hold, in fact, {\em uniformly} in $t$.
} 
\begin{equation}\label{eq:BDJ99}
\prob( L_n \leq l) = F_2\left(\frac{l-2\sqrt{n}}{n^{1/6}}\right) + o(1) \qquad (n\to\infty),
\end{equation}
uniformly in $l\in\N$; it will be called the {\em random matrix limit} of the length distribution throughout this paper since $F_2(t)$ denotes the Tracy--Widom distribution for $\beta=2$ (that is, the probability that in the soft-edge scaling limit of the Gaussian unitary ensemble (GUE) the scaled largest eigenvalue is bounded from above by $t$). This distribution can be evaluated numerically based on its representation either in terms of the Airy kernel determinant \cite{MR1236195} or in terms of the Painlevé-II transcendent \cite{MR1257246}; see Remark~\ref{rem:F2prime} and \cite{MR2895091} for details. 

As impressive as the use of the limit \eqref{eq:BDJ99} might look as a numerical approximation to the distribution of $L_n$ near its mode for larger $n$, cf. Fig.~\ref{fig:pdf100000}, there are two notable deficiencies: first, since the error term in \eqref{eq:BDJ99} is additive (i.e., w.r.t. absolute scale), the approximation is rather poor for $l \ll 2\sqrt{n}$; second, the convergence rate is rather slow, in fact conjectured to be just of the order $O(n^{-1/3})$; see \cite{arxiv.2205.05257} and the discussion below. Both deficiencies are well illustrated in Table~\ref{tab:1} for $n=20$ and in Fig.~\ref{fig:pdf1000} for $n=1000$.

\begin{table}[tbp]
\caption{The values of $\prob( L_n \leq l)$ and various of its approximations for $n=20$. The Monte Carlo simulation was run with $T=5\cdot 10^6$ samples. Clearly observable is a multiplicative (i.e., relative) error of already less than $1\%$ for the Stirling-type formula \eqref{eq:stirling} as well as an additive (i.e., absolute) error of order $4\cdot 10^{-4} \approx T^{-1/2}$ for Monte Carlo and
an additive error of order $10^{-1} \approx n^{-1/3}$ for the limit \eqref{eq:BDJ99} from random matrix theory.}
\label{tab:1}
{\footnotesize
\begin{tabular}{clllc}
\hline\noalign{\smallskip}
$l$ & $\;\prob( L_{20} \leq l)$ & Stirling-type \eqref{eq:stirling} & Monte-Carlo & random matrix \eqref{eq:BDJ99}\\
\noalign{\smallskip}\hline\noalign{\smallskip}
1 & $4.110\cdot 10^{-19}$&   \quad $4.119\cdot 10^{-19}$&   \quad\, $0.000$      &   $6.282\cdot 10^{-5}$   \\
2 & $2.698\cdot 10^{-9}$ &   \quad $2.703\cdot 10^{-9}$ &   \quad\, $0.000$      &   $1.422\cdot 10^{-3}$   \\ 
3 & $6.698\cdot 10^{-5}$ &   \quad $6.710\cdot 10^{-5}$ &   $7.240\cdot 10^{-5}$ &   $1.485\cdot 10^{-2}$   \\ 
4 & $1.090\cdot 10^{-2}$ &   \quad $1.092\cdot 10^{-2}$ &   $1.089\cdot 10^{-2}$ &   $8.014\cdot 10^{-2}$   \\ 
5 & $1.427\cdot 10^{-1}$ &   \quad $1.429\cdot 10^{-1}$ &   $1.428\cdot 10^{-1}$ &   $2.503\cdot 10^{-1}$   \\ 
6 & $4.841\cdot 10^{-1}$ &   \quad $4.846\cdot 10^{-1}$ &   $4.838\cdot 10^{-1}$ &   $5.079\cdot 10^{-1}$   \\ 
7 & $8.042\cdot 10^{-1}$ &   \quad $8.064\cdot 10^{-1}$ &   $8.040\cdot 10^{-1}$ &   $7.513\cdot 10^{-1}$   \\ 
8 & $9.521\cdot 10^{-1}$ &   \quad $9.581\cdot 10^{-1}$ &   $9.519\cdot 10^{-1}$ &   $9.041\cdot 10^{-1}$   \\ 
9 & $9.921\cdot 10^{-1}$ &   \quad $9.996\cdot 10^{-1}$ &   $9.920\cdot 10^{-1}$ &   $9.716\cdot 10^{-1}$   \\ 
\noalign{\smallskip}\hline
\end{tabular}}
\end{table}

\subsection*{A Stirling-type formula}

In this paper we suggest a different type of numerical approximation to the distribution of $L_n$ that enjoys the following advantages: (a) it has a small multiplicative (i.e., relative) error, (b) it has faster convergence rates, apparently even faster than \eqref{eq:BDJ99} with its first finite size correction term added, and (c) it is much faster to compute than Monte Carlo simulations. In fact, the distribution of $L_n$ for $n=10^5$, as shown in Fig.~\ref{fig:pdf100000}, exhibits an estimated  maximum additive error of less than $10^{-5}$ and took just about five seconds to compute; whereas Forrester and Mays \cite{arxiv.2205.05257} have recently reported a computing time of about 14 hours to generate $T = 5\cdot 10^6$ Monte Carlo trials for this $n$; the error of such a simulation is expected to be of the order $1/\sqrt{T} \approx 4\cdot 10^{-4}$. 

Specifically, we use Hayman's generalization \cite{Hayman56} of Stirling's formula for {\em $H$-admissible} functions; for expositions see \cite{MR2483235,MR1373678,Wong89}. For simplicity, as is the case here for $f(z)=f_l(z)$, assume that
\[
f(z) = \sum_{n=0}^\infty a_n z^n \qquad (z \in \C)
\]
is an entire function with positive coefficients $a_n$ and consider the real auxiliary functions
\[
a(r) = r \frac{d}{dr} \log f(r),\qquad b(r) = r \frac{d}{dr} a(r)\qquad (r>0).
\]
If $f$ is $H$-admissible, then for each $n\in \N$ the equation $a(r_n) = n$ has a unique solution $r_n > 0$ such that $b(r_n)>0$ and the following generalization of Stirling's formula\footnote{A generalization of Stirling's classical formula, indeed: for the $H$-admissible function $f(z)=e^z$, cf. Thm.~\ref{thm:hayman} Criterion II.g, we have $a_n=1/n!$, $b(r)=a(r)=r$, $r_n=n$ and \eqref{eq:hayman} specifies to
\[
\frac{1}{n!} = \frac{e^n}{n^n \sqrt{2\pi n}}(1+ o(1)) \qquad (n\to \infty).
\]
As in Table~\ref{tab:1}, the error is already below $1\%$ for $n$ as small as $n=20$.} holds true:
\begin{equation}\label{eq:hayman}
a_n = \frac{f(r_n)}{r_n^n \sqrt{2\pi\, b(r_n)}}(1+ o(1)) \qquad (n\to\infty).
\end{equation}
We observe that the error is multiplicative here.
In Thm.~\ref{thm:main} we will prove, using some theory of entire functions, the $H$-admissibility of the generating functions $f_l$. Hence the Stirling-type formula \eqref{eq:hayman} applies without further ado to their coefficients 
$\prob(L_n \leq l)/n!$. Since the error is multiplicative, nothing changes if we multiply the approximation by $n!$ and we get
\begin{equation}\label{eq:stirling}
\prob(L_n \leq l) = \frac{n! \cdot f_l(r_{l,n})}{r_{l,n}^n \sqrt{2\pi\, b_l(r_{l,n})}}(1+ o(1)) \qquad (n\to\infty),
\end{equation}
where we have labeled all quantities when applied to $f=f_l$ by an additional index $l$. An approximation to $\prob(L_n=l)$ is then obtained simply by taking differences. The power of these approximations, if used as a numerical tool even for $n$ as small as $n=20$, is illustrated in Table~\ref{tab:1} and Figs.~\ref{fig:pdf100000}--\ref{fig:error}, as well as in Table~\ref{tab:2} below.

\begin{figure}[tbp]
\includegraphics[width=0.455\textwidth]{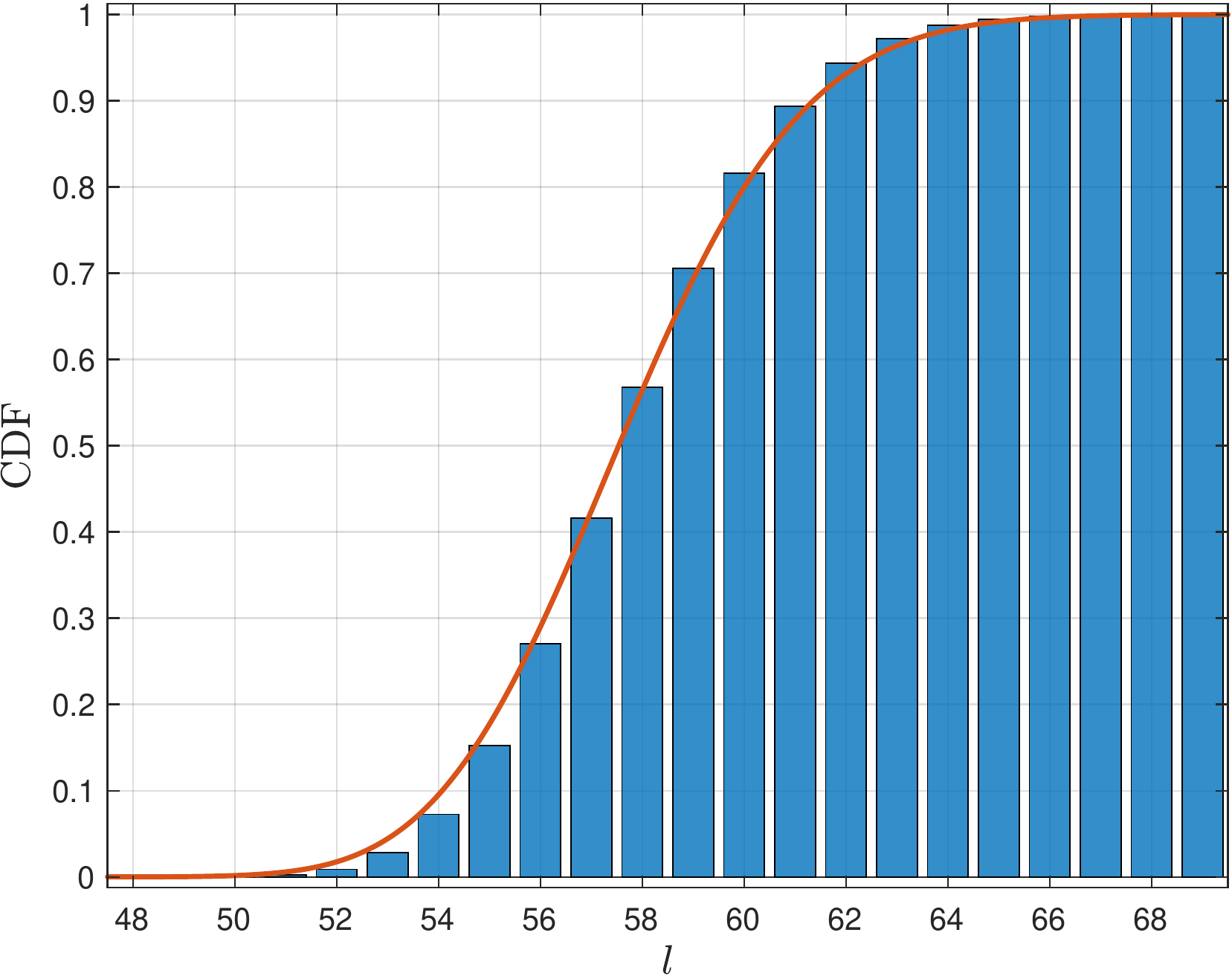}\hfil
\includegraphics[width=0.455\textwidth]{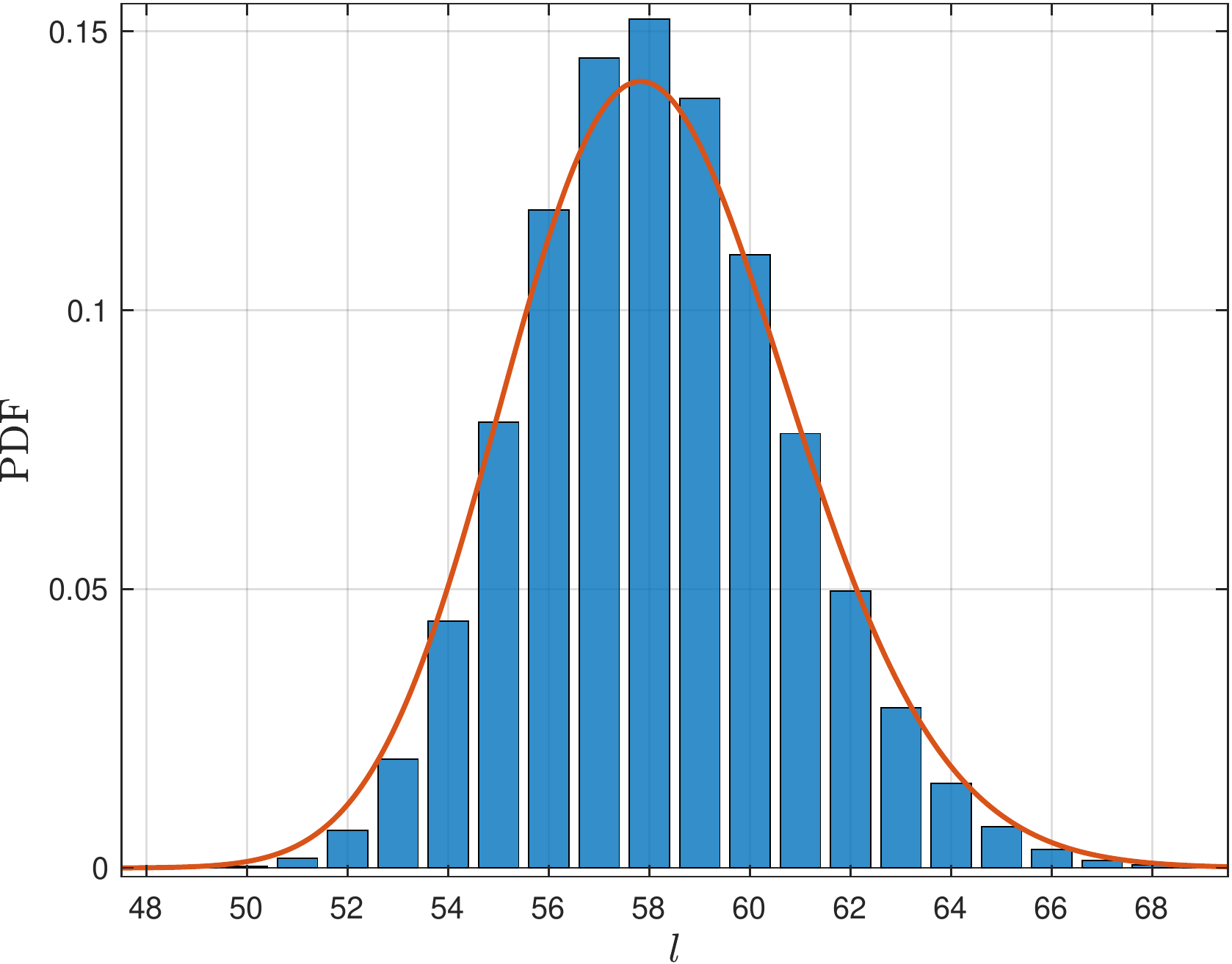}
\caption{{\footnotesize Display of the notable inaccuracy of the random matrix limit \eqref{eq:BDJ99} for $n=1000$ (see the contrast with Fig.~\ref{fig:pdf100000} for $n=10^5$). The discrete distribution of $L_n$ is shown near its mode vs. the random matrix limit given by the leading order terms in~\eqref{eq:CDFexpansion} and \eqref{eq:PDFexpansion} (solid red line). Left: CDF $\prob( L_n \leq l)$; right: PDF $\prob( L_n = l)$. The exact values of the distribution of $L_n$ and their approximation by the Stirling-type formula~\eqref{eq:stirling} differ just by additive errors of the order $10^{-4}$ (see Fig.~\ref{fig:error}), which is well below plotting accuracy.}}
\label{fig:pdf1000}
\end{figure}

\subsection*{Numerical evaluation of the generating function}

For the Stirling-type formula \eqref{eq:stirling} to be easily accessible in practice, we require an expression for $f_l(r)$ that can be numerically evaluated, for $r>0$, in a stable, accurate, and efficient fashion. Since the direct evaluation of the Toeplitz determinant \eqref{eq:Gessel} is numerically highly unstable, and has a rather unfavorable  complexity of $O(l^3)$ for larger $l$, we look for alternative representations. One option---used in \cite{MR2754188} to numerically extract $\prob(L_n\leq l)$ from $f_l(z)$ by Cauchy integrals over circles in the complex plane that are centered at the origin with the same radius $r_{l,n}$ as in \eqref{eq:stirling}---is the machinery, cf., e.g., \cite{MR1935759}, to transform Toeplitz determinants into Fredholm determinants which are then amenable for the numerical method developed in \cite{MR2600548}. However, since we need the values of the generating function $f_l(r)$ for real $r>0$ only, there is a much more efficient option, which comes from yet another connection to random matrix theory.

To establish this connection we first note that an exponentially generating function $f(r)$ of a sequence of probability distributions has a probabilistic meaning if multiplied by $e^{-r}$: a process called {\em Poissonization}. Namely, if the draws from the different permutation groups are independent and if we take $N_r \in \N_0 :=\{0,1,2,3,\ldots\}$ to be a further independent random variable with Poisson distribution of intensity $r\geq 0$, we see that
\begin{equation}\label{eq:poissonization}
\prob(L_{N_r} \leq l) = \sum_{n=0}^\infty \frac{e^{-r} r^n}{n!} \prob(L_n \leq l) = e^{-r} f_l(r) 
\end{equation}
is, for fixed $r\geq 0$, the cumulative probability distribution of the composite discrete random variable~$L_{N_r}$. On the other hand, for fixed $l\in \N$, also $e^{-r} f_l(r)$ turns out to be a probability distribution w.r.t. the continuous variable $r\geq 0$: specifically, in terms of precisely the Toeplitz determinant~\eqref{eq:Gessel}, Forrester and Hughes~\cite[Eq.~(3.33)]{MR1303076} arrived in 1994 at the representation
\begin{equation}\label{eq:forrester}
 E^{\text{(hard)}}_2(0; [0,4r],l) = e^{-r} f_l(r).
\end{equation}
Here, $E^{\text{(hard)}}_2(0; [0,t],l)$ denotes the probability that, in the hard-edge scaling limit of the Laguerre unitary ensemble (LUE) with parameter $l$, the smallest eigenvalue is bounded from below by $t\geq 0$. Now, the point here is that this distribution can be evaluated numerically, stable and accurate with a complexity that is largely independent of $l$, based on two alternative representations: either in terms of the Bessel kernel determinant \cite{MR1236195} or in terms of the Jimbo--Miwa--Okamoto $\sigma$-form of the Painlevé-III transcendent \cite{MR1266485}; see \cite{MR2895091} for details. We will show in Sect.~\ref{sect:numerical} that the auxiliary functions $a_l(r)$ and $b_l(r)$ fit into both frameworks, too.

\begin{figure}[tbp]
\includegraphics[width=0.465\textwidth]{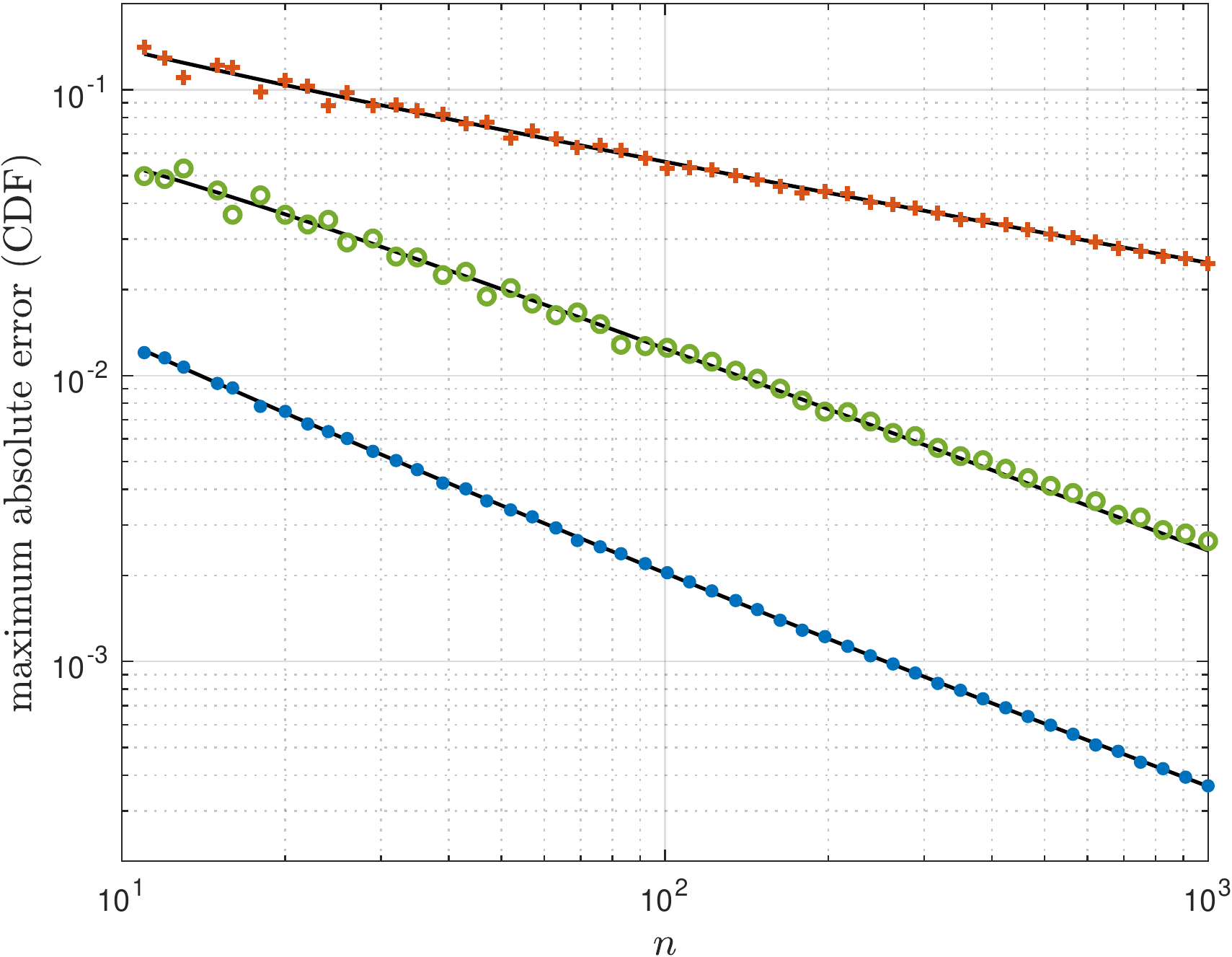}\hfil
\includegraphics[width=0.465\textwidth]{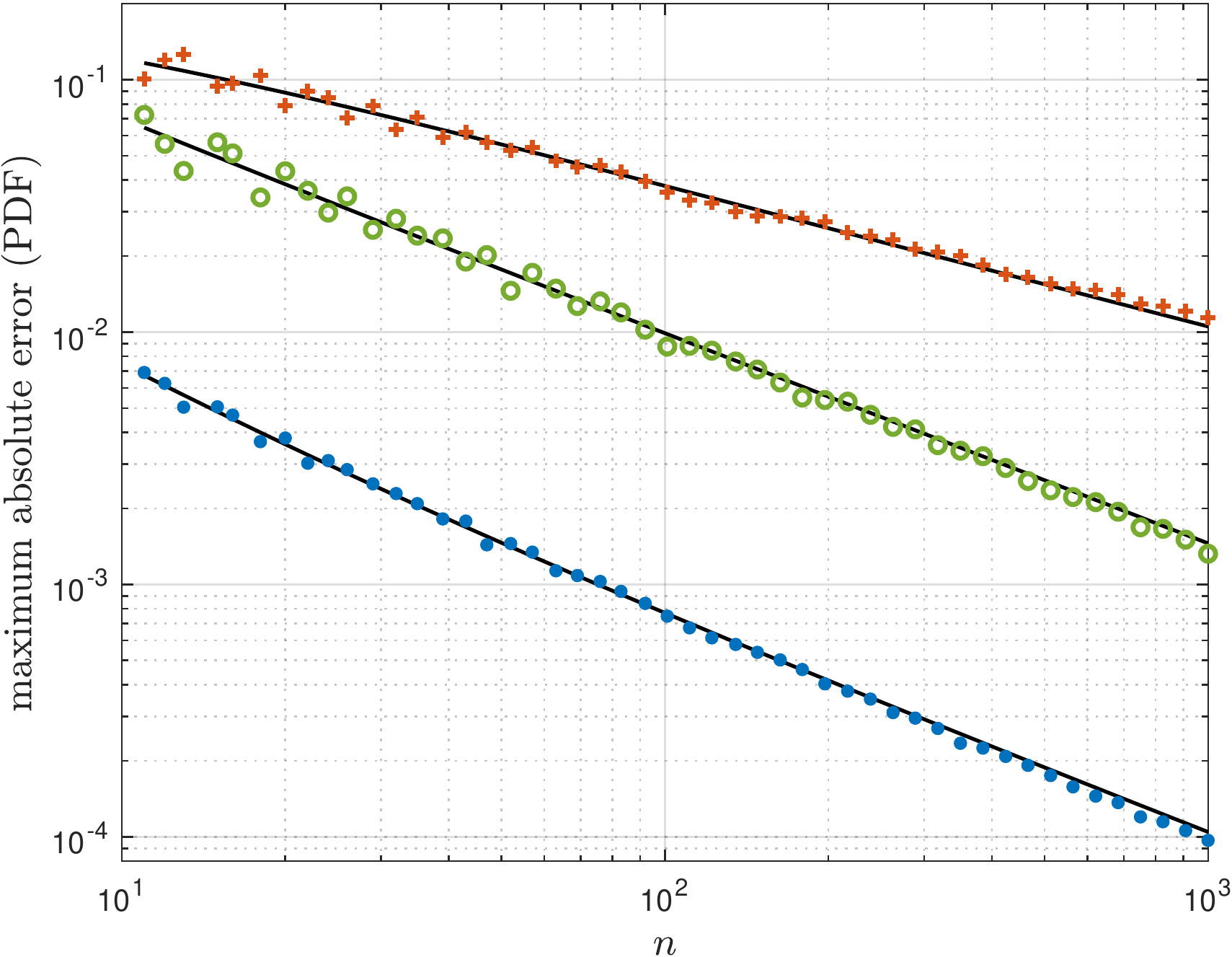}
\caption{{\footnotesize Maximum absolute (i.e., additive) errors of various approximations to (left panel) the CDF $\prob(L_n\leq l)$ and (right panel) the PDF $\prob(L_n= l)$ in a double logarithmic scaling, based on tabulated exact values up to $n=1000$, cf. Sect.~\ref{subsect:exact}; solid lines are fits of the form $c_1 n^{-\alpha_1} + c_2 n^{-\alpha_2} + c_3 n^{-\alpha_3}$ (CDF) and $n^{-1/6}\cdot(c_1 n^{-\alpha_1} + c_2 n^{-\alpha_2} + c_3 n^{-\alpha_3})$ (PDF) to the points in display. Red $+$: error of leading order terms (random matrix limit) in~\eqref{eq:CDFexpansion}, \eqref{eq:PDFexpansion}; $\alpha = (\tfrac13,\tfrac23,1)$. Green~$\circ$: error of expansions~\eqref{eq:CDFexpansion}, \eqref{eq:PDFexpansion} truncated after the first finite size correction term; $\alpha=(\tfrac23, 1, \tfrac43)$, where $F_{2,1}$ has been approximated as in Fig.~\ref{fig:1st}. Blue $\bullet$: error of the Stirling-type formula \eqref{eq:stirling}; $\alpha = (\tfrac23, 1, \tfrac43)$.}}
\label{fig:error}
\end{figure}

\subsection*{Finite size corrections to the random matrix limit}

In a double logarithmic scaling, a plot of the additive errors (taking the maximum w.r.t. $l\in\{1,2,\ldots,n\}$) in approximating the distribution $\prob(L_n\leq l)$ by either the random matrix limit \eqref{eq:BDJ99} or by the Stirling-type formula~\eqref{eq:stirling} exhibits nearly straight lines; see Fig.~\ref{fig:error} for $n$ between $10$ and $1000$.
Fitting the data in display to a model of the form $c_1 n^{-\alpha_1} + c_2 n^{-\alpha_2} + c_3 n^{-\alpha_3}$ with simple triples $\alpha=(\alpha_1, \alpha_2, \alpha_3)$ of rationals strongly suggests that, uniformly in $l \in \{1,2,\ldots,n\}$ as $n\to \infty$,
\begin{subequations}\label{eq:errorterm}
\begin{align}
\prob( L_n \leq l) &= F_2\left(\frac{l-2\sqrt{n}}{n^{1/6}}\right) + O(n^{-1/3}),\label{eq:BDJLandauO}
\\*[2mm]
\prob( L_n \leq l) &= \frac{n! \cdot f_l(r_{l,n})}{r_{l,n}^n \sqrt{2\pi\, b_l(r_{l,n})}} + O(n^{-2/3}).\label{eq:StirlingLandauO}
\end{align}
\end{subequations}
The approximation order (and the size of the implied constant) in \eqref{eq:StirlingLandauO} is much better than the one in  \eqref{eq:BDJLandauO} so that the Stirling-type formula can be used to reveal the structure of the $O(n^{-1/3})$ term in the random matrix limit. In fact, as the error plot in Fig.~\ref{fig:error} suggests and we will more carefully argue in Sect.~\ref{sect:CDFexpansion}, this can even be iterated yet another step and we are led to the specific conjecture\footnote{Note that $l\in \N$ is always a discrete variable in this paper, so there is no need for taking integer parts here.}
\begin{equation}\label{eq:CDFexpansion}
\prob(L_n\leq l) = F_2(t_l) + n^{-1/3} F_{2,1}(t_l) + n^{-2/3} F_{2,2}(t_l) + O(n^{-1}), \quad t_l := \frac{l-2\sqrt{n}}{n^{1/6}},
\end{equation}
as $n\to\infty$, uniformly in $l\in \N$. Compelling evidence for the existence of the functions $F_{2,1}$ and $F_{2,2}$  is given in the left panels of Figs.~\ref{fig:1st} and \ref{fig:2nd}.\footnote{Based on Monte-Carlo simulations, Forrester and Mays \cite[Eq.~(1.10) and Fig.~7]{arxiv.2205.05257} were recently also led to conjecture an expansion of the form
\[
\prob(L_n\leq l) = F_2(t_l) + n^{-1/3} F_{2,1}(t_l) + \cdots.
\]
However, the substantially larger errors of Monte Carlo simulations as compared to the Stirling-type formula~\eqref{eq:stirling} would inhibit them from getting, in reasonable time, much evidence about the next finite size correction term.}

\begin{figure}[tbp]
\includegraphics[width=0.475\textwidth]{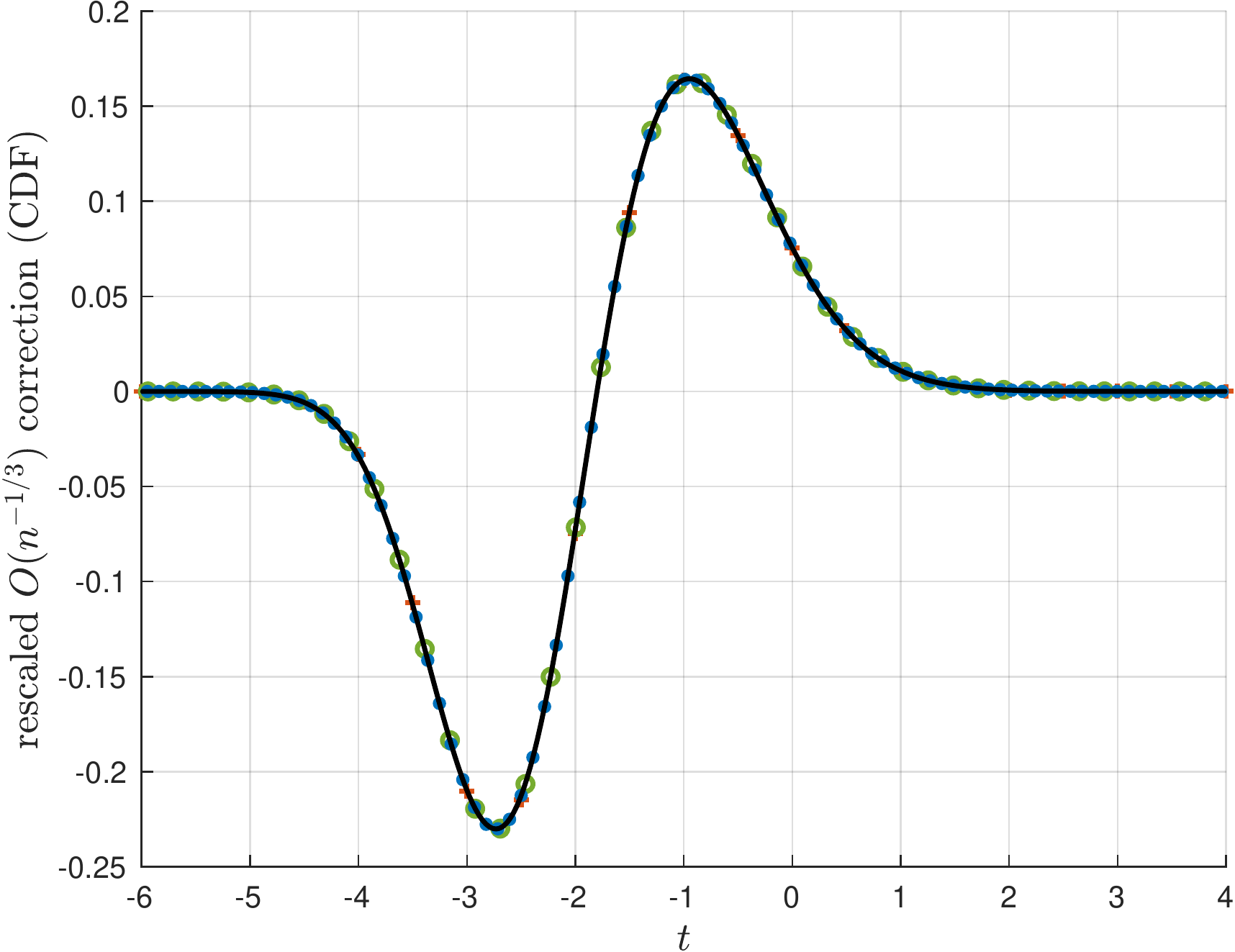}\hfil
\includegraphics[width=0.475\textwidth]{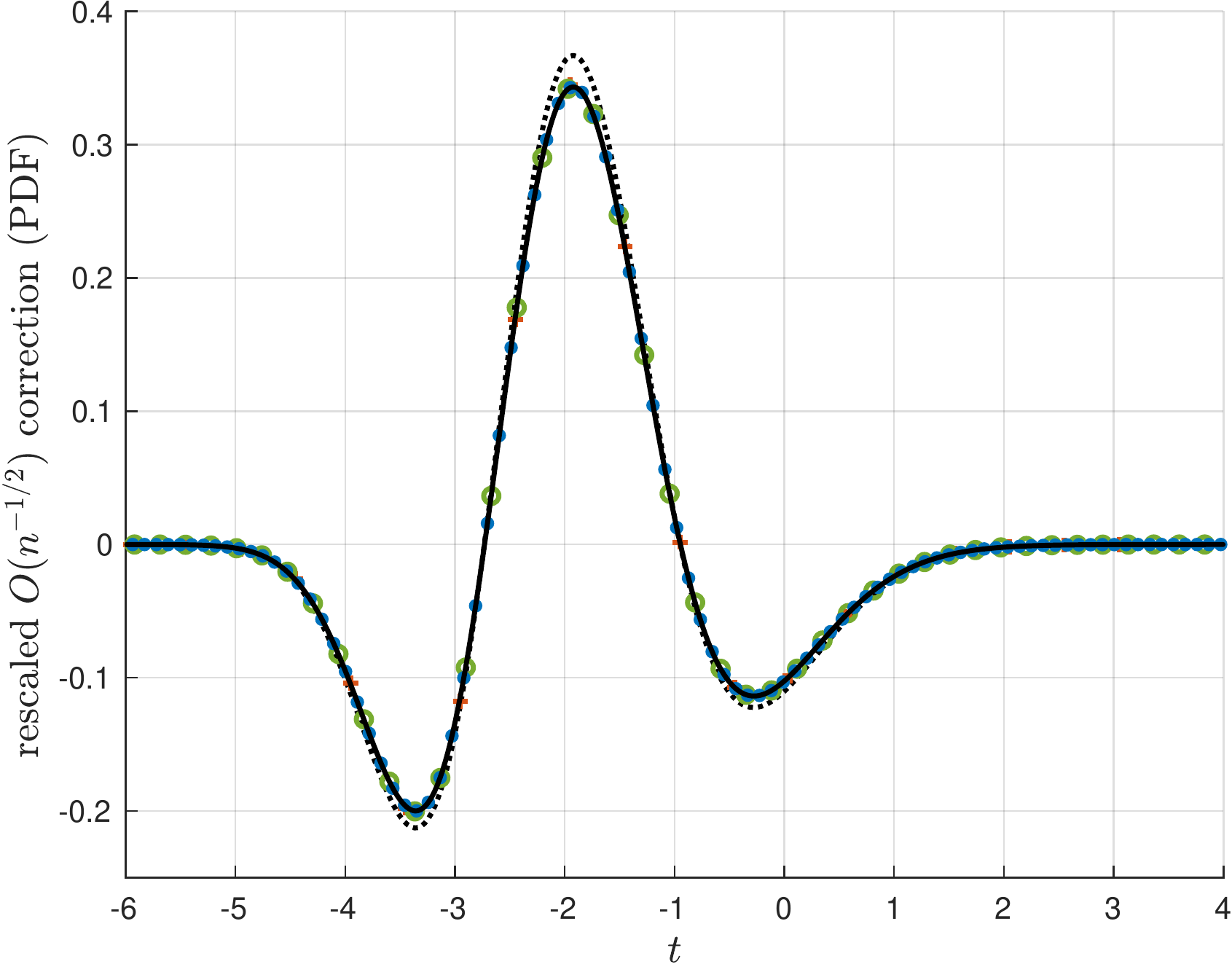}
\caption{{\footnotesize Rescaled differences between the distributions of $L_n$ and their expansions truncated after the leading order term (i.e., the random matrix limit)---see \eqref{eq:CDFexpansion} for the CDF resp.~\eqref{eq:PDFexpansion} for the PDF; data points (to avoid clutter just every $5^\text{th}$ is displayed) have been calculated using the Stirling-type formula~\eqref{eq:stirling} for $n=10^6$ (red $+$), $n=10^8$ (green~$\circ$), $n=10^{10}$ (blue $\bullet$). Left: CDF errors rescaled by $n^{1/3}$, horizontal axis is $t=(l-2\sqrt{n})/n^{1/6}$, cf. \cite[Fig.~7]{arxiv.2205.05257} for a similar figure with data from Monte Carlo simulations for $n=2\cdot 10^4$ and $n=10^5$. The solid line is a polynomial $\tilde F_{2,1}(t)$ of degree~$64$ fitted to the $836$ data points for $n=10^{10}$ with $-8\leq t \leq 10$; it approximates $F_{2,1}(t)$ in that interval. Right: PDF errors rescaled by $n^{1/2}$, horizontal axis is $t=(l-\frac12-2\sqrt{n})/n^{1/6}$. The solid line displays the function $\tilde F_{2,1}'(t)+F_2'''(t)/24$ as an approximation of $F_{2,1}'(t)+F_2'''(t)/24$, with the polynomial $\tilde F_{2,1}(t)$ taken from the left panel. The dotted line shows the term $\tilde F_{2,1}'(t)$ only.}}
\label{fig:1st}
\end{figure}

We note that a corresponding expansion\footnote{Expansions of probability distributions are sometimes called {\em Edgeworth expansions} in reference to the classical one for the central limit theorem. In random matrix theory a variety of such expansions have been studied: e.g., for the soft-edge scaling limits of GUE/LUE \cite{MR2233711} and GOE/GSE \cite{MR2492622}, for the hard-edge scaling limit of LUE \cite{MR3513610} and L$\beta$E \cite{MR4019586}, for the bulk scaling limit of CUE/COE/CSE \cite{MR3647807}. For the the hard-to-soft edge transition limit of LUE see the expansion \eqref{eq:BaikJenkins} and its discussion.}
 for the Poissonization \eqref{eq:poissonization} of the length distribution was studied by Baik and Jenkins \cite[Eq.~(25)]{MR3161478} (using the machinery of Riemann--Hilbert problems up to an error of order $O(r^{-1/2})$) and by Forrester and Mays \cite[Eqs.~(1.18), (2.29)]{arxiv.2205.05257} (using Fredholm determinants), who obtained the expansion, as $r\to \infty$ for bounded $t_l^*$:
\begin{subequations}\label{eq:BaikJenkins}
\begin{equation}\label{eq:BaikJenkinsExpansion} 
\prob(L_{N_r} \leq l)  = F_2(t_l^*) + r^{-1/3} F_{2,1}^*(t_l^*) + O(r^{-2/3}), \quad t_l^* := \frac{l-2\sqrt{r}}{r^{1/6}},
\end{equation}
with the explicit functional form (identified by means of Painlevé representations)
\begin{equation}
F_{2,1}^*(t) = -\frac{1}{10} \left( F_2''(t) + \frac{t^2}{6} F_2'(t)\right).
\end{equation}
\end{subequations}
Though \eqref{eq:BaikJenkinsExpansion} adds to the plausibility of the expansion \eqref{eq:CDFexpansion}, the de-Poissonization lemma of Johansson \cite[Lemma~2.5]{MR1618351} and its commonly used variants (see \cite{MR1682248,MR3468920,MR3468738}) would not even allow us to deduce from \eqref{eq:BaikJenkins} the existence of the term $F_{2,1}(t)$, let alone to obtain its functional form.

{\footnotesize\begin{note}[added in proof]  On the other hand, by inserting the Poissonized expansion \eqref{eq:BaikJenkins} (and the induced expansions of the quantities $b(r)$ and $r_{l,n}$) into the Stirling-type formula \eqref{eq:StirlingLandauO} with its conjectured error of order $O(n^{-2/3})$,  we are led to the conjecture
\begin{equation}\label{eq:F21}
F_{2,1}(t) =  F_{2,1}^*(t) - \frac{1}{2} F_2''(t) = -\frac{1}{10} \left( 6 F_2''(t) + \frac{t^2}{6} F_2'(t)\right).
\end{equation}
This functional form is in perfect agreement with the data displayed in Fig.~\ref{fig:1st}; see Footnote~\ref{foot:F21accuracy} and Remark~\ref{rem:F21mean} for further numerical evidence. Details will be given in a forthcoming paper of the author \cite{forth}, where the expression \eqref{eq:F21} for $F_{2,1}(t)$ (as well as one for $F_{2,2}(t)$) is also obtained by a complex-analytic modification (related to $H$-admissibility) of the de-Poissonization process.
\end{note}}

We will argue in Sect.~\ref{sect:mean} that the expansion \eqref{eq:CDFexpansion} of the length distribution allows us to derive an expansion of the expected value of $L_n$, specifically
\begin{gather}
\Erw(L_n) = 2\sqrt{n} + \mu_0 n^{1/6} + \frac{1}{2} + \mu_1 n^{-1/6} + \mu_2 n^{-1/2} + O(n^{-5/6}),\label{eq:mean}\\*[0.5mm]
\mu_0 = \int_{-\infty}^\infty t \, F_2'(t)\,dt = -1.77108\,68074\cdots,\notag\\*[0.5mm]
\mu_1 = \int_{-\infty}^\infty t \, F_{2,1}'(t)\,dt = 0.06583\,238\cdots, \quad
\mu_2 = \int_{-\infty}^\infty t \, F_{2,2}'(t)\,dt = 0.26122\,27\cdots.\notag
\end{gather}
Similarly, we will derive in Sect.~\ref{sect:var} an expansion of the variance of $L_n$ of the form
\begin{align}
\Var(L_n) &= \nu_0 n^{1/3} + \nu_1 + \nu_2 n^{-1/3} + O(n^{-2/3}),\label{eq:var0}\\*[0.5mm]
\nu_0 &= \int_{-\infty}^\infty t^2 F_2'(t)\,dt - \mu_0^2 = 0.81319\,47928\cdots,\notag\\*[0.5mm]
\nu_1 &= \int_{-\infty}^\infty t^2 F_{2,1}'(t)\,dt +\frac{1}{12} -2\mu_0\mu_1 = -1.20720\,507\cdots,\notag\\*[0.5mm]
\nu_2 &= \int_{-\infty}^\infty t^2 F_{2,2}'(t)\,dt -\mu_1^2 - 2\mu_0\mu_2 = 0.56715\,6\cdots.\notag
\end{align}
The values  of  $\mu_0$ and $\nu_0$ are the known values of mean and variance of the Tracy--Widom distri\-bu\-tion~$F_2$, cf. \cite[Table~10]{MR2895091}. (The leading parts of \eqref{eq:mean} up to $\mu_0n^{1/6}$  and of~\eqref{eq:var0} up to~$\nu_0n^{1/3}$ had been established previously by Baik, Deift and Johansson \cite[Thm.~1.2]{MR1682248}.)

\section{$H$-Admissibility of the Generating Function and its Implications}\label{sect:hayman}

\subsection{$H$-admissible functions}
For simplicity we restrict ourselves to entire functions. We refrain from displaying the rather lengthy technical definition of $H$-admissibility,\footnote{Since we consider entire functions only, $H$-admissibility is here understood to hold in all of $\C$.} which is difficult to be verified in practice and therefore seldomly directly used. Instead, we start by collecting some usefuls facts and criteria from Hayman's original paper \cite{Hayman56}:\footnote{Interestingly, the powerful criterion in part III (which is \cite[Thm.~XI]{Hayman56}) is missing from the otherwise excellent expositions \cite{MR2483235,MR1373678,Wong89} of $H$-admissibility.}

\begin{theorem}[Hayman 1956]\label{thm:hayman} Let $f(z) = \sum_{n=0}^\infty a_nz^n$ and $g(z)$
be entire functions and let $p(z)$ denote a polynomial with real coefficients. 

\medskip

{\rm I.} If $f$ is $H$-admissible, then:
\begin{itemize}\itemsep5pt
\item[a.] $f(r) > 0$ for all sufficiently large $r>0$, so that in particular the auxiliary functions\footnote{In terms of differential operators we have $r \frac{d}{dr} = \frac{d}{d\log r}$.}
\begin{equation}\label{eq:aux}
a(r) = r \frac{d}{dr} \log f(r),\qquad b(r) = r \frac{d}{dr} a(r)
\end{equation}
are well defined there;
\item[b.] for $r>0$ as in {\rm I.a} there is $\log f(r)$ strictly convex in $\log r$, $a(r)$ strictly monotonically increasing, and $b(r)>0$ such that $a(r), b(r)\to\infty$ as $r\to\infty$; in particular, for large integers $n$ there is a unique $r_n>0$ that solves $a(r_n)=n$, it is $r_n\to\infty$ as $n\to\infty$;
\item[c.] if the coefficients $a_n$ of $f$ are all positive, then {\rm I.b} holds for all $r>0$;
\item[d.] as $r\to\infty$, uniformly in $n \in \N_0$, 
\begin{equation}\label{eq:CLT}
\frac{a_n r^n}{f(r)} = \frac{1}{\sqrt{2\pi b(r)}}\left(\exp\left(-\frac{(n-a(r))^2}{2b(r)}\right)+ o(1)\right).
\end{equation}
\end{itemize}

\medskip

{\rm II.} If $f$ and $g$ are $H$-admissible, then:
\begin{itemize}\itemsep5pt
\item[e.] $f(z)g(z)$, $e^{f(z)}$ and $f(z) + p(z)$ are $H$-admissible;
\item[f.] if the leading coefficient of $p$ is positive, $f(z)p(z)$ and $p(f(z))$ are $H$-admissible;
\item[g.] if the Taylor coefficients of $e^{p(z)}$ are eventually positive, $e^{p(z)}$ is $H$-admissible.  
\end{itemize}

\medskip

{\rm III.} If $f$ has genus zero\footnote{By definition, an entire function $f$
has genus zero if it is a polynomial or if it can be represented as a convergent infinite product of the form
\[
f(z) = c z^m \prod_{n=1}^\infty \Big(1- \frac{z}{z_n}\Big)\qquad (z \in \C),
\]
where $c\in \C$ is a constant, $m\in\N_0$ is the order of the zero at $z=0$, and $z_1,z_2,\ldots$ is the sequence of the non-zero zeros, where each one is listed as often as multiplicity requires. 
}
with, for some $\delta >0$, at most finitely many zeros in the sector $|\arg z|\leq \frac{\pi}{2}+\delta$ and
satisfies {\rm I.a} such that $b(r)\to\infty$ as $r\to\infty$, then $f$ is $H$-admissible.
\end{theorem}

Obviously, the Stirling-type formula \eqref{eq:hayman} is obtained from the approximation result \eqref{eq:CLT} by just inserting the particular choice $r=r_n$. 

\begin{remark}
We observe that, if $a_n\geq 0$, Eq.~\eqref{eq:CLT} has an interesting probabilistic content:\footnote{Note that the particular case $f(z)=e^z$ (which is $H$-admissible by Thm.~\ref{thm:hayman}.II.g) specifies to the well-known normal approximation of the Poisson distribution for large intensities---which is a simple consequence of the central limit theorem if we observe that the sum of $k$ independent Poisson random variates of intensity $\rho$ is one of intensity $r = \rho k$.} as a distribution in the discrete variable $n\in\N_0$, the Boltzmann\footnote{We follow the terminology in the theory of Boltzmann samplers \cite{MR2095975}, a framework for the random generation of combinatorial structures. Note that mean and variance of the Boltzmann propabilities $a_n r^n/f(r)$ are exactly the auxiliary functions  $a(r)$ and $b(r)$ as defined in \eqref{eq:aux}, cf. \cite[Prop.~2.1]{MR2095975}.} probabilities $a_n r^n/f(r)$ associated with an $H$-admissible entire function $f$ are, for large intensities $r>0$, approximately normal with mean $a(r)$ and variance $b(r)$; see the right panel of Fig.~\ref{fig:CLT} for an illustrative example using the generating function $f_5(z)$. The additional freedom that is provided in the normal approximation \eqref{eq:CLT} by the uniformity w.r.t. $n$ will be put to good use in Sect.~\ref{sect:regev}.
\end{remark}

\begin{figure}[tbp]
\includegraphics[width=0.5\textwidth]{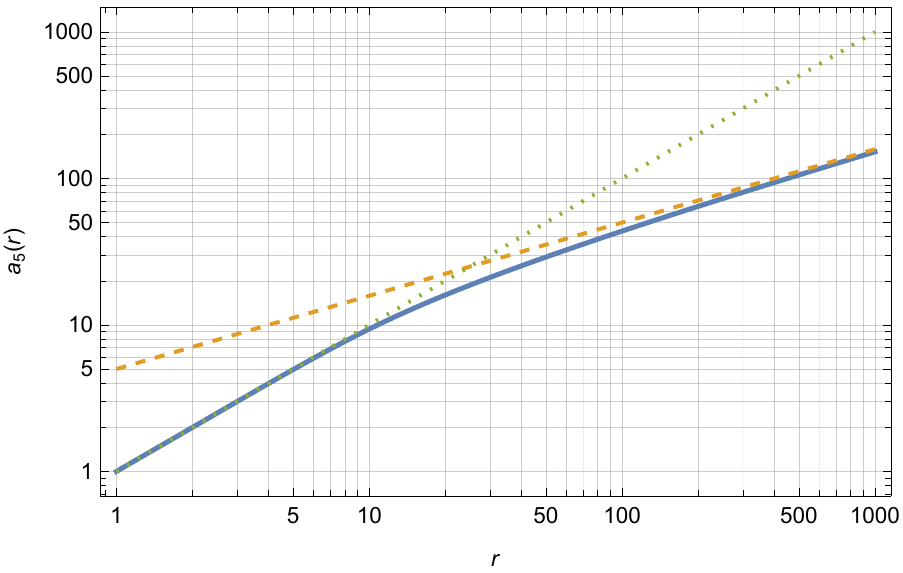}\hfil
\includegraphics[width=0.4125\textwidth]{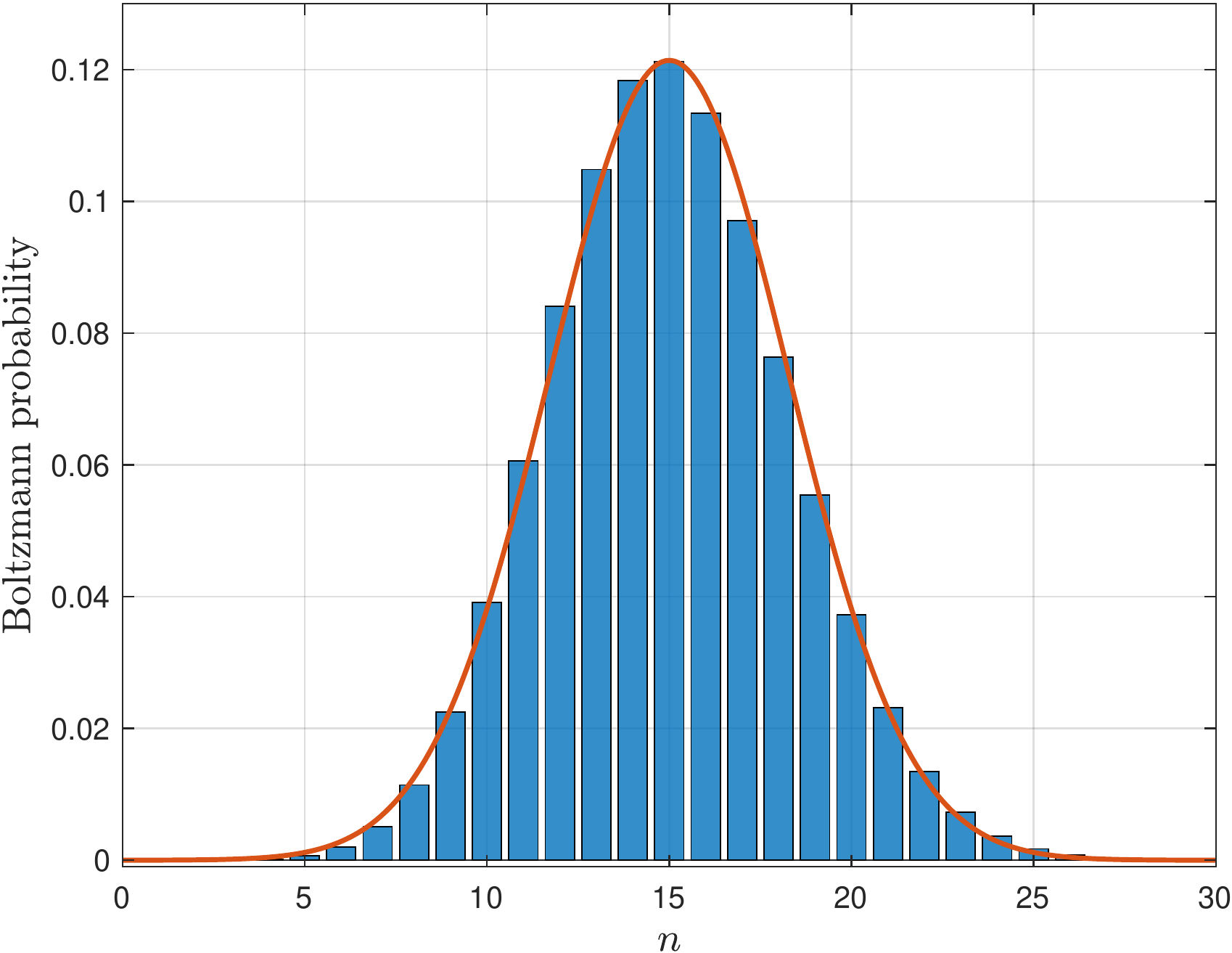}
\caption{{\footnotesize Left: the auxiliary function $a_5(r)$ (blue solid line) associated with the generating function $f_5(z)$, together with the asymptotics $a_5(r)= r + O(r^6)$ as $r\to 0$ (green dotted line) and $a_5(r)= 5 r^{1/2} + O(1)$ as $r\to \infty$ (red dashed line). Right: Illustration of the approximation \eqref{eq:CLT} of the Boltzmann probabilities (blue bars) associated with the generating function $f_5(z)$ for intensity $r_{5,15}\approx 18.23$, cf. the notation in \eqref{eq:stirling}. The normal distribution (red solid line) has mean $a_5(r_{5,15}) = 15$ (cf. the left panel) and variance $b_5(r_{5,15})\approx 10.80$.}}
\label{fig:CLT}
\end{figure}

The classification of entire functions (by quantities such as genus, order, type, etc.) and their distribution of zeros is deeply related to the analysis of their essential singularity at $z=\infty$. For the purposes of this paper, the following simple criterion is actually all we need. The proof uses some theory of entire functions, which can be found, e.g., in \cite{MR589888}.

\begin{lemma}\label{lem:born} Let $f(z)$ be an entire function of exponential type with positive Maclaurin coefficients. If there are constants $c, \tau, \nu > 0$ such that there holds, for the principal branch of the power function and for each $0 < \delta\leq \frac\pi2$,
the asymptotic expansion\footnote{See Remark~\ref{rem:bigo} for the uniformity implied by the notation $O(z^{-1})$ as $z\to\infty$ while $|\!\arg z|\leq \tfrac{\pi}{2}-\delta$.}
\begin{equation}\label{eq:Dexpansion}
D(z):=f(z^2) = cz^{-\nu} e^{\tau z}(1 + O(z^{-1}))\qquad (z\to\infty,\; |\!\arg z|\leq \tfrac{\pi}{2}-\delta),
\end{equation}
then $f$ is $H$-admissible. For $r\to\infty$ the associated auxiliary functions $a(r)$ and $b(r)$ satisfy
\begin{equation}\label{eq:auxexpansion}
a(r) = \frac{\tau}{2}r^{1/2} - \frac{\nu}{2} + O(r^{-1/2}), \qquad
b(r) = \frac{\tau}{4}r^{1/2}  + O(r^{-1/2}),
\end{equation}
and the solution $r_n$ of $a(r_n)=n$ satisfies
\begin{equation}\label{eq:rnexpansion}
r_n =\frac{4n}{\tau^{2}} (n+\nu) + O(1) \qquad (n\to\infty).
\end{equation}
\end{lemma} 
\begin{proof} The expansion \eqref{eq:Dexpansion} is equivalent to
\begin{equation}\label{eq:fasympt}
f(z) = cz^{-\nu/2} e^{\tau z^{1/2}}(1 + O(z^{-1/2}))\qquad (z\to\infty,\; |\!\arg z|\leq \pi-2\delta),
\end{equation}
which readily implies:
\begin{itemize}
\item since $\delta>0$ is arbitrary, $f$ has the Phragmén--Lindelöf indicator
\[
\limsup_{r\to\infty} \frac{\log|f(r e^{i\theta})|}{r^{1/2}} = \tau\cos(\theta/2) \qquad (-\pi< \theta < \pi),
\]
so that $f$ has order $\tfrac12$ and type $\tau$, hence genus zero; 
\item for sufficiently large $R = R_{\delta}>0$, there are {\em no} zeros $z$ of $f$ with
\[
|z|\geq R,\qquad |\!\arg z| \leq \pi -2\delta.
\]
\end{itemize}
Since the Maclaurin coefficients of $f$ are positive we have $f(r)>0$ for $r>0$ and the auxiliary functions $a(r)$, $b(r)$ in \eqref{eq:aux} are well-defined for $r>0$. In fact, both functions can be analytically continued into the domain of uniformity of the expansion \eqref{eq:fasympt} and by differentiating this expansion (which is, because of analyticity, legitimate by a theorem of Ritt, cf. \cite{MR0435697}) we obtain \eqref{eq:auxexpansion}; this implies, in particular, $b(r) \to \infty$ as $r\to\infty$.
Thus, all the assumptions of Thm.~\ref{thm:hayman}.III are satisfied and $f$ is shown to be $H$-admissible.
\end{proof}

\subsection{Singularity analysis of the generating function at \boldmath$z=\infty$\unboldmath}

Establishing an expansion of the form \eqref{eq:Dexpansion} for $D_l(z)=f_l(z^2)$ as given by \eqref{eq:Gessel}, that is to say, for
the Toeplitz determinant
\begin{equation}\label{eq:toeplitz}
D_l(z) =\det_{j,k=1}^l I_{j-k}(2z),
\end{equation}
suggests to start with the expansions (valid for all $0<\delta\leq \frac\pi2$, see \cite[p.~251]{MR0435697})
\begin{subequations}\label{eq:Iexpansion}
\begin{align}
I_m(z) &\sim \frac{e^{z}}{(2\pi z)^{1/2}} \sum_{n=0}^\infty (-1)^n \frac{A_n(m)}{z^n}\qquad \qquad (z\to\infty,\; |\!\arg z|\leq \tfrac{\pi}{2}-\delta),\\*[2mm]
A_n(m) &= \frac{(4m^2-1^2)(4m^2-3^2)\cdots(4m^2-(2n-1)^2)}{n! \,8^n}.
\end{align}
\end{subequations}
This does not yield \eqref{eq:Dexpansion} at once, as there could be, however unlikely it would be, eventually a catastrophic cancellation of {\em all} of the expansion terms when being inserted into the determinant expression defining $D_l(z)$. For the specific cases $l=1,2,\ldots,8$
 a computer algebra system shows that exactly the first $l-1$ terms of the expansion \eqref{eq:Iexpansion} mutually cancel each other in forming the determinant, and we get by this approach\footnote{Odlyzko \cite[Ex.~10.9]{MR1373678} reports that he and Wilf had used this approach, before 1995, for small $l$ in the framework of the method of “subtraction of singularities” in asymptotic enumeration. He states the expansions for $D_4(z)$ and $D_5(z)$, cf. \cite[Eqs.~(10.30)/(10.39)]{MR1373678}, with a misrepresented constant factor in $D_5(z)$, though. No attempt, however, was made back then to guess the general form.} the expansions
\begin{align*}
D_1(z) &= \frac{e^{2 z}}{2 \pi^{1/2} z^{1/2}}(1+O(z^{-1})),
& D_2(z) &=  \frac{e^{4 z}}{8 \pi  z^2}(1+O(z^{-1})),\\*[1mm]
D_3(z) &=   \frac{e^{6z}}{32 \pi ^{3/2} z^{9/2}}(1+O(z^{-1})),
& D_4(z) &=   \frac{3 e^{8 z}}{256 \pi^2 z^8}(1+O(z^{-1})),\\*[1mm]
D_5(z) &=   \frac{9 e^{10 z}}{1024 \pi ^{5/2}z^{25/2}}(1+O(z^{-1})),
& D_6(z) &=   \frac{135 e^{12 z}}{8192 \pi ^3z^{18}}(1+O(z^{-1})),\\*[1mm]
D_7(z) &=   \frac{6075 e^{14 z}}{65536 \pi ^{7/2}z^{49/2}}(1+O(z^{-1})),
& D_8(z) &= \frac{1913625 e^{16 z}}{1048576 \pi ^4z^{32}}(1+O(z^{-1})).\\*[-3mm]
\end{align*}
All of them, inherited from \eqref{eq:Iexpansion}, are valid as $z\to\infty$ while $|\!\arg z|\leq \tfrac{\pi}{2}-\delta$ with the uniformity content implied by the symbol $O(z^{-1})$. From these instances, in view of \eqref{eq:Iexpansion}
and the multi\-linearity of the determinant, we guess that
\[
D_l(z) = c_l \frac{e^{2lz}}{(4\pi z)^{l/2} (2z)^{l(l-1)/2}} (1+ O(z^{-1})) \qquad \qquad (z\to\infty,\; |\!\arg z|\leq \tfrac{\pi}{2}-\delta)
\]
and observe 
\[
c_1, c_2, c_3, c_4, c_5, c_6, c_7, c_8, \ldots = 1, 1, 2, 12, 288, 34560, 24883200, 125411328000,\ldots\;.
\]
Consulting the OEIS\footnote{\url{https://oeis.org/A000178}} (On-Line Encyclopedia of Integer Sequences) suggests the coefficients to be generally of the form 
\[
c_l = 0!\cdot 1!\cdot 2! \, \cdots\,  (l-1)!\,.
\]
Though this is very likely to hold for all $l\in\N$---a fact that would at once yield the $H$-admissiblity of all the generating function $f_l$ by Lemma~\ref{lem:born}---a proof seems to be elusive along these lines, but see Remark~\ref{rem:RHP} for a remedy. 

Inspired by the fact that the one-dimensional Laplace's method easily gives the leading order term in \eqref{eq:Iexpansion} when applied to the Fourier representation 
\begin{equation}\label{eq:besselint}
I_m(2z) = \frac{1}{2\pi} \int_{-\pi}^{\pi}e^{2z\cos \theta} e^{-im \theta}\,d\theta\qquad (z\in \C, m \in \Z),
\end{equation}
we represent the Toeplitz determinant $D_l(z)$ in terms of a multidimensional integral and study
the limit $z\to \infty$ by the multidimensional Laplace method discussed in the Appendix. In fact, \eqref{eq:besselint} shows that the symbol of the Toeplitz determinant $D_l(z)$ is $\exp(2z\cos\theta)$ and a classical formula of Szeg\H{o}'s \cite[p.~493]{Sz1915} from 1915, thus gives, without further calculation, the integral representation\footnote{This induces, see \eqref{eq:forrester}, 
an integral representation of the distribution $E^{\text{(hard)}}_2(0; [0,s],l)$ which has been derived in 1994 by Forrester \cite{MR1271945} using generalized hypergeometric functions defined in terms of Jack polynomials.}
\begin{equation}\label{eq:rains}
D_l(z) = \frac{1}{(2\pi)^l \,l!} \int_{-\pi}^{\pi}\cdots\int_{-\pi}^{\pi} e^{2z \sum_{j=1}^l \cos\theta_j}\cdot \big|\Delta(e^{i\theta_1},\ldots,e^{i\theta_l})\big|^2 \,d\theta_1 \cdots d\theta_l,
\end{equation}
where 
\[
\Delta(w_1,\ldots,w_l) := \prod_{j > k} (w_j- w_k)
\]
denotes the Vandermonde determinant of the complex numbers $w_1,\ldots,w_l$.
\begin{remark} By Weyl's integration formula on the unitary group $U(l)$, cf.~\cite[Eq.~{1.5.89}]{MR2105088}, the integral \eqref{eq:rains} can be written as
\[
D_n(z) = \E_{U\in U(l)} e^{z \tr(U+U^*)},
\]
where the expectation $\E$ is taken with respect to the Haar measure. Without any reference to \eqref{eq:Gessel}, this form was derived in 1998 by Rains \cite[Cor.~4.1]{Rains98} directly
from the identity 
\[
\prob(L_n\leq l) = \E_{U\in U(l)} (|\tr U|^{2n}),
\]
which he had obtained most elegantly from the representation theory of the symmetric group. \end{remark}

We are now able to prove our main theorem.

\begin{theorem}\label{thm:main} For each $0< \delta \leq \pi/2$ and $l\in \N$ there holds the asymptotic expansion
\[
D_l(z) = f_l(z^2) =\frac{0!\cdot 1!\cdot 2! \, \cdots\,  (l-1)!\cdot e^{2l z}}{(2\pi)^{l/2}(2z)^{l^2/2}}(1+ O(z^{-1})) \qquad (z\to\infty,\; |\!\arg z|\leq \tfrac{\pi}{2}-\delta).
\]
Thus, by Lemma~\ref{lem:born}, the generating functions $f_l(z)$ are $H$-admissible and their auxiliary functions satisfy, as $r\to\infty$,
\begin{equation}\label{eq:aux_large_r}
a_l(r) = lr^{1/2} - \tfrac14 l^2 + O(r^{-1/2}),\qquad b_l(r) = \tfrac12 l r^{1/2} + O(r^{-1/2}).
\end{equation}
\end{theorem}
\begin{proof} We write \eqref{eq:rains} in the form
\[
e^{-lz} D_l(z/2) = \frac{1}{(2\pi)^l \,l!} \int_{-\pi}^{\pi}\cdots\int_{-\pi}^{\pi} e^{-z \sum_{j=1}^l (1-\cos\theta_j)}\cdot \big|\Delta(e^{i\theta_1},\ldots,e^{i\theta_l})\big|^2 \,d\theta_1 \cdots d\theta_l.
\]
The phase function of this multidimensional integrand, that is to say
\[
S(\theta_1,\ldots,\theta_l) :=  \sum_{j=1}^l (1-\cos\theta_j), 
\]
takes it minimum $S(\theta_*)=0$ at $\theta_* = 0$ with the expansion $S(\theta_1,\ldots,\theta_l) = \frac{1}{2}\theta^T\theta + O(|\theta|^4)$ as $\theta \to 0$,
where $|\cdot|$ denotes Euclidian length. Likewise we get for the non-exponential factor\footnote{This factor is zero at $\theta_*=0$, so that Hsu's variant \eqref{eq:Hsu} of Laplace's method, which is the one predominantly found in the literature, does not yield the leading term of $D_l(z)$. That we have to expand the non-exponential factor up to degree $l(l-1)$ for the first non-zero contribution to show up, corresponds to the mutual cancellation of the leading terms of \eqref{eq:Iexpansion} when being inserted into the Toeplitz determinant that defines $D_l(z)$.}
\begin{align*}
\big|\Delta(e^{i\theta_1},\ldots,e^{i\theta_l})\big|^2 &= \prod_{j>k} |e^{i\theta_j}-e^{i\theta_j}|^2  = 
\prod_{j>k} \left|i\theta_j-i\theta_j + O(|\theta|^2)\right|^2 \\*[1mm] 
&= \Delta(\theta_1,\ldots,\theta_l)^2 + O(|\theta|^{l(l-1)+1})\qquad (\theta \to 0),
\end{align*}
where the degree of the homogeneous polynomial $\Delta(\theta_1,\ldots,\theta_l)^2$ is $l(l-1)$. 

Therefore, by the multidimensional Laplace method as given in Corollary~\ref{cor:Laplace} (see also formula \eqref{eq:laplacezero} following it) we obtain immediately
\[
e^{-lz} D_l(z/2) = c_l \frac{(2\pi)^{l/2} z^{-\frac{l+l(l-1)}{2}}}{(2\pi)^l} (1+O(z^{-1}))  \qquad (z\to\infty,\; |\!\arg z|\leq \tfrac{\pi}{2}-\delta)
\]
with 
\begin{equation}\label{eq:Selberg}
c_l := \frac{1}{l!}\frac{1}{(2\pi)^{l/2}} \int_{\R^l} e^{- \theta^T\theta/2}\cdot \big|\Delta(\theta_1,\ldots,\theta_l)\big|^2 \,d\theta 
= 0!\cdot 1!\cdot 2! \, \cdots\,  (l-1)!,
\end{equation}
where the evaluation of this multiple integral is well-known in random matrix theory, e.g., as a consequence of Selberg's integral formula, cf.~\cite[Eq.~(2.5.11)]{MR2760897}.
\end{proof}

\begin{remark}\label{rem:RHP} By \eqref{eq:forrester}, Thm.~\ref{thm:main} implies
\[
E^{\text{(hard)}}_2(0; [0,s],l) = \frac{0!\cdot 1!\cdot 2! \, \cdots\,  (l-1)!}{(2\pi)^{l/2}} \cdot s^{-l^2/4} e^{-s/4 + l s^{1/2}}(1+O(s^{-1/2}))\qquad (s\to\infty),
\]
an asymptotics first rigorously proven, using Riemann--Hilbert problem machinery, by Deift, Krasovsky and Vasilevska \cite{MR2806560} in 2010. Besides that our proof is much simpler, their result, which is for real $s>0$ only, would by itself not suffice to establish the $H$-admissibility of the generating function $f_l(z)$; one would have to complement it with the arguments given above for expanding the Toeplitz determinant \eqref{eq:toeplitz} based on the expansions \eqref{eq:Iexpansion} of the modified Bessel functions. However, their result is more general in another respect: it covers parameters $\alpha \in \C$ of the LUE with $\Re \alpha > -1$ instead of just $l \in \N$; the superfactorial factor $0!\cdot 1!\cdot 2! \, \cdots\,  (l-1)!$ is then to be replaced by $G(1+\alpha)$, where $G(z)$ is the Barnes $G$-function.\footnote{For real $\alpha >-1$, Tracy and Widom~\cite{MR1266485} had conjectured this asymptotics in 1994 based on a guess of the connection formula \eqref{eq:PIIIconnection} for the Painlevé transcendent \eqref{eq:sigmaPIIIinitial} and a numerical exploration of the constant factor. In the same year Forrester~\cite{MR1271945} confirmed this to be true for $\alpha\in\N$ by sketching an argument that, basically, uses the idea underlying the multidimensional Laplace method in the proof of Thm.~\ref{thm:main}. So, Corollary~\ref{cor:Laplace} can be used to spell out the details there, and for the generalization to $\beta$-ensembles sketched in \cite[p.~608]{MR2641363}.} 
\end{remark}

We complement the large $r$ expansion  \eqref{eq:aux_large_r} of the auxiliary functions with their expansions as $r\to 0^+$,
which are simple consequences of elementary combinatorics.

\begin{lemma}\label{lem:aux_small_r} The auxiliary functions of the generating function $f_l$ satisfy, as $r\to 0^+$,
\begin{equation}\label{eq:aux_small_r}
a_l(r) = r - \frac{r^{l+1}}{l! \cdot (l+1)!} + O(r^{l+2}),\qquad b_l(r) = r-\frac{r^{l+1}}{l!\cdot l!} + O(r^{l+2}).
\end{equation}
\end{lemma}
\begin{proof}
Because of $L_n \leq n$ and since there is just one permutation $\sigma$ with $L_n=n$, we get
\[
\prob(L_n \leq l) = \begin{cases}
1 & \quad n \leq l,\\*[1mm]
1- \frac{1}{(l+1)!}  &\quad n = l+1.
\end{cases}
\] 
This implies, by truncating the power series of $f_l$ at order $l+1$,
\[
f_l(z) = 1 + z + \frac{z^2}{2!} + \cdots + \frac{z^l}{l!} + \frac{1-\frac{1}{(l+1)!}}{l!}z^{l+1} + O(z^{l+2}) 
= e^z - \frac{z^{l+1}}{((l+1)!)^2} + O(z^{l+2}). 
\]
Logarithmic differentiation of the power series thus yields, as $z\to \infty$,
\[
z \frac{f'(z)}{f(z)} = z - \frac{z^{l+1}}{l! (l+1)!} + O(z^{l+2}), \quad z \frac{d}{ dz} \left( z \frac{f'(z)}{f(z)}\right) = 
z - \frac{z^{l+1}}{(l!)^2} + O(z^{l+2}),
\]
and the results follow from specializing to $z = r>0$.
\end{proof}

The left panel of Fig.~\ref{fig:CLT} visualizes the expansions \eqref{eq:aux_large_r} and \eqref{eq:aux_small_r} for $a_5(r)$. Apparently, as a function of $\log r$, the auxiliary $\log a_l(r)$ interpolates monotonically, concavely, and from below between the following two extremal regimes:

\begin{itemize}\itemsep2pt
\item $\log r$ as $r\to 0^+$, which reflects, by the proof of Lemma~\ref{lem:aux_small_r}, the regime $l \geq n$, and\\*[-3mm]
\item $\frac12\log r + \log l$ as $r\to \infty$, which reflects the regime $l \ll n^{1/4}$ (see Sect.~\ref{sect:regev}).
\end{itemize}

\noindent It is this seamless interpolation between the two regimes $l \geq n$ and $l \ll n^{1/4}$ that helps to understand the observed uniformity 
of the Stirling-type formula w.r.t. $l$, cf. \eqref{eq:StirlingLandauO}.

\begin{table}[tbp]
\caption{The Stirling-type formula \eqref{eq:stirling} vs. Regev's formula \eqref{eq:regev} for $l=5$. The exact values of $\#\{\sigma: L_n(\sigma) \leq 5\} = n!\cdot \prob(L_n\leq 5)$ were computed using the holonomic $4$-term recursion w.r.t. $n$ (see \cite[p.~468]{MR1336855}). Note that the Stirling-type formula gives $4$ correct digits already for $n=160$; Regev's formula starts to deliver $4$ correct digits only for $n$ as large as $n\approx 2\cdot 10^5$, where one is just about to enter the regime $l\ll n^{1/4}$. Including the Gaussian factor \eqref{eq:gaussiancorrection} helps to improve the accuracy a little, relaxing the bound on $l$ to about $l\leq 2 n^{1/4}$.}
\label{tab:2}
{\footnotesize
\begin{tabular}{rllll}
\hline\noalign{\smallskip}
$n$ & $\#\{\sigma: L_n(\sigma) \leq 5\}$ & Stirling-type \eqref{eq:stirling} & Regev \eqref{eq:regev}  & $e^{-l^4/16n}\; \times$ \eqref{eq:regev} \\
\noalign{\smallskip}\hline\noalign{\smallskip}
$20$       & \quad $3.472\cdot 10^{17}$        & $3.477\cdot 10^{17}$       &   $2.159\cdot 10^{18}$       &   $3.062\cdot 10^{17}$       \\
$40$       & \quad $1.836\cdot 10^{42}$        & $1.837\cdot 10^{42}$       &   $4.794\cdot 10^{42}$       &   $1.805\cdot 10^{42}$       \\ 
$80$       & \quad $5.915\cdot 10^{94}$        & $5.918\cdot 10^{94}$       &   $9.681\cdot 10^{94}$       &   $5.941\cdot 10^{94}$       \\ 
$160$      & \quad $1.260\cdot 10^{203}$       & $1.260\cdot 10^{203}$      &   $1.617\cdot 10^{203}$      &   $1.267\cdot 10^{203}$      \\ 
$320$      & \quad $1.630\cdot 10^{423}$       & $1.630\cdot 10^{423}$      &   $1.848\cdot 10^{423}$      &   $1.636\cdot 10^{423}$      \\ 
$640$      & \quad $9.287\cdot 10^{866}$       & $9.287\cdot 10^{866}$      &   $9.891\cdot 10^{866}$      &   $9.305\cdot 10^{866}$      \\ 
$1\,280$   & \quad $1.124\cdot 10^{1\,758}$    & $1.124\cdot 10^{1\,758}$   &   $1.160\cdot 10^{1\,758}$   &   $1.125\cdot 10^{1\,758}$   \\ 
$2\,560$   & \quad $6.434\cdot 10^{3\,543}$    & $6.434\cdot 10^{3\,543}$   &   $6.536\cdot 10^{3\,543}$   &   $6.437\cdot 10^{3\,543}$   \\ 
$200\,000$ & \quad $2.383\cdot 10^{279\,530}$  & $2.383\cdot 10^{279\,530}$ &   $2.383\cdot 10^{279\,530}$ &   $2.383\cdot 10^{279\,530}$ \\ 
\noalign{\smallskip}\hline
\end{tabular}}
\end{table}

\subsection{A new proof of Regev's asymptotic formula for fixed \boldmath$l$\unboldmath}\label{sect:regev} 

A simple closed form expression in terms of $n$ and $l$ is obtained by studying the asymptotics, as $n \to \infty$ for fixed $l$, of the Stirling-type formula \eqref{eq:stirling} itself---or even easier yet, because of its
added flexibility, of Hayman's normal approximation \eqref{eq:CLT} for a suitable choice of $r$. As tempting as it might appear, however, this stacking of asymptotics leads, first, to  a considerable loss of approximation power for small $n$, cf. Table~\ref{tab:2}, and, second, to a lack of uniformity w.r.t. $l$ since the result is effectively conditioned to the constraint $l\ll n^{1/4}$.

To avoid notational clutter, we suppress the index $l$ from the generating function $f_l$, its auxiliaries $a_l$, $b_l$ and from the radius $r_{l,n}$. Solving $a(r_n)=n$ yields, by \eqref{eq:rnexpansion}, the expansion
\begin{equation}\label{eq:rstirlingasymp}
r_n = \frac{n^2}{l^2} + \frac{n}{2} + O(1) \qquad (n\to\infty);
\end{equation}
which suggests to plug its leading order term $r^*_n := n^2/l^2$ into \eqref{eq:CLT}. Thm.~\ref{thm:main} gives, as $n\to\infty$,
\begin{gather*}
f(r_n^*) = \frac{0!\cdot 1!\cdot 2! \, \cdots\,  (l-1)!}{(2\pi)^{l/2} 2^{l^2/2}} \cdot e^{2n} \left(\frac{l}n\right)^{l^2/2} (1+ O(n^{-1})),\\*[1mm]
a(r_n^*) = n - \tfrac{1}{4}l^2 + O(n^{-1}),\qquad b(r_n^*) = \tfrac12 n + O(n^{-1}).
\end{gather*}
The Gaussian term in \eqref{eq:CLT} has thus the expansion
\begin{equation}\label{eq:gaussiancorrection}
\exp\left(-\frac{(n-a(r_n^*))^2}{2b(r_n^*)}\right) = e^{-l^4/16n}\big(1+ O(n^{-2})\big) = 1 - \frac{l^4}{16n} + O(n^{-2}),
\end{equation}
which indicates that we can expect $r_n^*$ to deliver a quality of approximation comparable to~$r_n$ (which corresponds to using the Stirling-type formula) only if $l\ll n^{1/4}$; see Table~\ref{tab:2} for an illustrative example. Altogether Hayman's normal approximation \eqref{eq:CLT} gives, choosing $r=r_n^*$,
\[
\#\{\sigma: L_n(\sigma) \leq l\} = n!\cdot \prob(L_n\leq l) =  \frac{0!\cdot 1!\cdot 2! \, \cdots\,  (l-1)!}{(2\pi)^{l/2} 2^{l^2/2}} \cdot \frac{(n!)^2 \left(\frac{e}{n}\right)^{2n} l^{2n + l^2/2}}{\sqrt{\pi}\, n^{(l^2+1)/2}}(1+ o(1))
\] 
as $n\to\infty$. Wrapping up by using Stirling's formula in the form
\[
(n!)^2 \left(\frac{e}{n}\right)^{2n}  = 2\pi n \,(1+O(n^{-1}))
\]
we have thus given a new proof of Regev's formula \cite[Eq.~(4.5.2)]{MR625890}:
\begin{equation}\label{eq:regev}
\#\{\sigma: L_n(\sigma) \leq l\}  = \frac{0!\cdot 1!\cdot 2! \, \cdots\,  (l-1)! \cdot l^{2n + l^2/2}}{(2\pi)^{(l-1)/2}(2n)^{(l^2-1)/2}}(1+o(1))\qquad (n\to\infty).
\end{equation}

\begin{remark} The fixed $l$ asymptotics \eqref{eq:regev} was first proved by Regev~\cite{MR625890} in 1981, cf.~\cite[Thm.~7]{MR2334203}. His
delicate and rather long\footnote{Though Regev
studies, with a real parameter $\beta>0$, the more general combinatorial sums
\[
S_l^{(\beta)}(n):=\sum_{\lambda \,\vdash n\,:\, l_\lambda \leq l} d_\lambda^\beta,
\]
this generality adds only marginally to the complexity of his proof. In its final step he refers to the same instance of Selberg's integral, cf. \cite[Eq.~(2.5.11)]{MR2760897}, that we have used to obtain \eqref{eq:Selberg} in the specific case $\beta=2$.}
 proof proceeds, first, by identifying the leading contributions to the finite sum \eqref{eq:hook} using Stirling's formula, and then, after trading exponentially decaying tails (the basic idea of Laplace's method), by approximating the sum by a multidimensional integral which, finally, leads to the evaluation of Selberg's integral \eqref{eq:Selberg}. 
\end{remark}

\section{Numerical Evaluation of the Generating Function and its Auxiliaries}\label{sect:numerical}

The numerical evaluation of the Stirling-type formula \eqref{eq:stirling} requires the evaluation of the generating function $f_l(r)$ and its auxiliaries $a_l(r)$, $b_l(r)$ for real $r>0$. This will be based on the representation \eqref{eq:forrester}. That is to say, by writing
\begin{subequations}\label{eq:aux_in_rmt}
\begin{equation}
g_l(s) :=E^{\text{(hard)}}_2(0; [0,s],l), \qquad v_l(s) := - s \frac{d}{ds} \log g_l(s), \qquad u_l(s):= s v_l'(s),
\end{equation}
for the functions from random matrix theory, we obtain 
\begin{equation}
f_l(r) = e^r g_l(4r),\qquad a_l(r) = r - v_l(4r),\qquad b_l(r) = r - u_l(4r).
\end{equation}
\end{subequations}
As is common in the discussion of the LUE, we generalize this by replacing $l\in\N$ with a real parameter $\alpha>-1$. Dropping the index altogether we write, briefly, just $g(s)$, $v(s)$, and $u(s)$.

\subsection{Evaluation in terms of \boldmath$\sigma$-Painlevé-III\unboldmath}

The work of Tracy and Widom \cite{MR1266485} shows
\[
g(s) = \exp\left(-\int_0^s v(x)\,\frac{dx}{x}\right)\qquad (s\geq 0),
\]
where $v(s)$ satisfies a Jimbo--Miwa--Okamoto $\sigma$-form of the Painlevé-III equation (related to the Hamiltonian formulation PIII$'$ in the work of Okamoto; cf. \cite[§8.2]{MR2641363}), i.e., the nonlinear second order differential equation
\begin{equation}\label{eq:sigmaPIII}
(x v'')^2 - (\alpha v')^2 + v'(v-x v')(4v'-1) = 0\qquad (x>0),
\end{equation}
subject to the following initial condition, which is consistent with \eqref{eq:aux_small_r} for $\alpha=l\in\N$:
\begin{equation}\label{eq:sigmaPIIIinitial}
v(x) = \frac{1}{\Gamma(\alpha+1)\Gamma(\alpha+2)} \left(\frac{x}4\right)^{\alpha+1} (1+ O(x)) \qquad (x\to 0^+).
\end{equation}
A numerical integration of the initial value problem gives approximations to $v(s)$ and $v'(s)$, thus also to $u(s)= s v'(s)$. As explained in 
\cite{MR2895091} a direct numerical integration has stability issues as the solution $v$ is a separatrix solution of the $\sigma$-Painlevé-III equation. It is therefore advisable to solve the differential equation numerically as an asymptotic boundary problem by supplementing the initial condition by its connection formula, that is the corresponding expansion for $x\to \infty$:
\begin{equation}\label{eq:PIIIconnection}
v(x) = \frac{x}{4} - \frac{\alpha}{2}x^{1/2} + \frac{\alpha^2}{4} + \frac{\alpha}{16}x^{-1/2} + \frac{\alpha^2}{16}x^{-1} + O(x^{-3/2})\qquad (x\to\infty).
\end{equation}
Note the consistency with \eqref{eq:aux_large_r} for $\alpha = l\in\N$; for general $\alpha>-1$ this connection formula was conjectured by Tracy and Widom \cite[Eq.~(3.1)]{MR1266485}, a rigorous proof is given in \cite{MR2806560}, cf. Remark~\ref{rem:RHP}. 

\subsection{Compiling a table of exact rational values}\label{subsect:exact}
As observed recently by Forrester and Mays \cite[Sec. 4.2]{arxiv.2205.05257}, substituting a truncated power series expansion of $v(x)$ into the $\sigma$-Painlevé-III equation~\eqref{eq:sigmaPIII} is a comparatively cost-efficient way\footnote{There are holonomic recurrences satisfied by $n!\cdot \prob(L_n\leq l)$ w.r.t. $n$; cf. the explicit formulae for $l=2,3$ in~\cite[p.~281]{Gessel90}, for $l=4$ in \cite[p.~556]{MR2334203} (the cases $l=3,5$ are misprinted there), for $l=5$ in \cite[p.~468]{MR1336855}; we have used the one for $l=5$ in Table~\ref{tab:2}. 
For $l>5$ the polynomial coefficients  quickly become unwieldy, though. \label{foot:recurrence}} to compile a table of the {\em exact} rational values of the distribution $\prob(L_n \leq l)$; they report to have done so up to $n=700$. 

We note that instead of dealing directly with \eqref{eq:sigmaPIII} in this fashion, it is of advantage to use an equivalent third-order differential equation belonging to the Chazy-I class,\footnote{In fact, this equation is obtained as the particular choice $c_1=c_2=c_4=c_5=c_8=c_9=0$, $c_3=-1$, $c_6=1/4$, $c_7=-l^2/4$, $f(x)=x$ in the full Chazy-I equation of the form discussed in \cite[Eq.~(A3)]{MR1752309}.} namely
\begin{equation}\label{eq:chazy}
   v_l''' +\frac{1}{x}v_l'' -\frac{6}{x}v_l'^2 +\frac{4}{x^2}v_lv_l' +\frac{x-l^2}{x^2} v_l'-\frac{1}{2 x^2}v_l = 0,
\end{equation}
which is obtained from differentiating \eqref{eq:sigmaPIII} w.r.t. $x$ and dividing the result by $2x^2v_l''^2$. Note that the Chazy-I equation \eqref{eq:chazy} is linear in the highest order derivative of $v_l$ and quadratic in the lower orders, whereas the $\sigma$-Painlevé-III equation \eqref{eq:sigmaPIII} is quadratic in the highest order derivative and cubic in the lower orders. Therefore,
substituting the expansion
\begin{subequations}\label{eq:chazyRecursion}
\begin{equation}
v_l(x) = \sum_{n=l+1}^\infty a_n x^n 
\end{equation}
into the Chazy-I equation \eqref{eq:chazy}  yields a much simpler recursive formula for the $a_n$, $n=l+1,\ldots$\,:
\begin{equation}
(n+1)(n^2-l^2)a_{n+1} + (n-\tfrac12)a_n - 2\sum_{m=l+1}^{n-l} m a_m \cdot (3(n-m)+1) a_{n+1-m} = 0,
\end{equation}
uniquely determining the coefficients $a_n$ from the initial value \eqref{eq:sigmaPIIIinitial}, that is, from
\begin{equation}
a_{l+1} = \frac{1}{4^{l+1}l!(l+1)!}.
\end{equation}
\end{subequations}
It is now a simple matter of truncated power series calculations in a modern computer algebra system to expand the generating function itself, 
\[
f_l(r) = \exp\left(r - \int_0^{4r} v_l(x)\frac{dx}{x}\right) = \exp\left( r - \sum_{n=l+1}^K \frac{4^n a_n}{n}r^n + O(r^{K+1})\right).
\]
Avoiding the overhead of reducing fractions and computing common denominators in exact rational arithmetic, we have used significance arithmetic with $\lceil 2.5\log_{10} (1000!) \rceil= 6420$ digits and subsequent rational reconstructions to compile a table\footnote{The table is available for download at \url{https://box-m3.ma.tum.de/f/7c4f8cb22f5d425f8cff/}. The tabulated values were checked, for $l=1,\ldots,5$, against the recurrences cited in Footnote~\ref{foot:recurrence} and, for $n-l=0,\ldots,20$, against an explicit formula by Goulden \cite[Cor.~3.4(a)]{MR1080995}---note the restriction on $l$ for it to hold true:
\[
\prob(L_n=l) =  \sum_{i,j,k \geq 0, i+j+k\leq n-l} \frac{(-1)^{i+j}n!}{i!j!k! (n-i-k)!(n-j-k)!} \qquad (l \geq (n-1)/2).
\]} of $\prob(L_n = l), 1\leq l \leq n$, up to $n = 1000$ (in just about $1.5$ hours CPU time using one core of a 3GHz Xeon server). This table is used in Figs.~\ref{fig:error} and \ref{fig:2nd} as well as Sect.~\ref{sect:mean} (note that Table~\ref{tab:1} could have been compiled with the values for up to $n=36$ that were tabulated in the work of Baer and Brock \cite{MR228216}).

\subsection{Evaluation in terms of Bessel kernel determinants and traces} In \cite{MR2600548} the author has shown that  Nyström's method for integral equations can be generalized to the numerical evaluation of Fredholm determinants. Thus, as advocated in \cite{MR2895091}, there is a stable and efficient numerical method to directly address the representation
\begin{subequations}
\begin{equation}
g(s) = \det(I- K)|_{L^2(0,s)},
\end{equation}
derived by Forrester \cite{MR1236195} in 1993, in terms of the Bessel kernel
\begin{equation}\label{eq:besselkernel}
K(x,y) := \frac{\smash[b]{J_\alpha(\sqrt{x}) \sqrt{y} J_{\alpha-1}(\sqrt{y}) - \sqrt{x} J_{\alpha-1}(\sqrt{x}) J_\alpha(\sqrt{y})} }{2(x-y)}.
\end{equation}
This numerical method was extended in \cite[Appendix]{MR3647807} to the evaluation of general terms involving determinants, traces, and resolvents of integral operators. The evaluation of the auxiliary functions $v(s)$, $u(s)$, as defined in \eqref{eq:aux_in_rmt}, is thus facilitated by the following theorem.
 
\begin{theorem}\label{thm:uv} Let $K(x,y)$ be a smooth kernel that induces an integral operator $K$ on $L^2(0,s)$ for all $s>0$ and define the derived kernel as
\[
K'(x,y) := K(x,y) + xK_x(x,y) + yK_y(x,y).
\]
Then, if we assume $g(s) = \det(I- K)|_{L^2(0,s)}>0$ for all $s>0$, there holds
\begin{align*}
v(s) &= - s \frac{d}{ds} \log g(s) = \tr \big ((I-K)^{-1} K' \big)|_{L^2(0,s)},\\*[2mm]
u(s) &= s v'(s)  = \tr \big( (I-K)^{-1} K'' + ((I-K)^{-1} K')^2\big)|_{L^2(0,s)}.
\end{align*}
\end{theorem}
\begin{proof} Rescaling integrals w.r.t. the measure $d\mu(y)=K(x,y)\,dy$ from being taken over the interval $(0,s)$ to $(0,1)$ induces a transformation of the kernel $K(x,y)$ according to
\[
K_s(x,y) = s K(sx,sy).
\]
This way we can keep the space $L^2(0,1)$ fixed while the kernels become dependent on the parameter $s>0$; in particular, then, there is no need to distinguish in notation between kernels and their induced integral operators. Now, if we denote differentiation w.r.t. to the parameter $s$ by a dot, we get 
\[
\dot{K}_s(x,y) = K'(sx,sy) = s^{-1} K_s'(x,y),
\]
and thus, by \cite[Lemma~1]{MR3647807}, the logarithmic derivative
\[
g'(s)/g(s) = -  \tr \big ((I-K_s)^{-1} \dot K_s \big) |_{L^2(0,1)}  = - s^{-1} \tr \big ((I-K)^{-1} K'  \big)|_{L^2(0,s)},
\]
which proves the asserted formula for $v(s) = -s g'(s)/g(s)$. Since, cf. \cite[Eq.~(2.4)]{MR1266485},
\[
\frac{d}{ds} (I-K_s)^{-1} = (I-K_s)^{-1} \dot K_s (I-K_s)^{-1},
\]
we get by the linearity of the trace
\[
u(s) = s v'(s) = \tr \big( (I-K_s)^{-1} K_s'' + ((I-K_s)^{-1} K_s')^2\big)|_{L^2(0,1)},
\]
which finishes the proof after a back-transformation to $L^2(0,s)$.
\end{proof}

For the Bessel kernel \eqref{eq:besselkernel} at hand we get the derived kernels
\begin{align*}
K'(x,y) &= \frac{1}{4}J_\alpha(\sqrt{x})J_\alpha(\sqrt{y}),\\*[1mm]
K''(x,y) &= (1-\alpha) K'(x,y) + \frac{1}{8}\Big(J_\alpha(\sqrt{x}) \sqrt{y} J_{\alpha-1}(\sqrt{y}) + \sqrt{x} J_{\alpha-1}(\sqrt{x}) J_\alpha(\sqrt{y})\Big).
\end{align*}
We observe that both, $K'$ and $K''$, induce integral operators of finite rank, namely
\begin{gather*}
K' = \phi\otimes \phi,\quad K'' = (1-\alpha)\, \phi \otimes \phi + \frac{1}{2} (\phi \otimes \psi + \psi \otimes \phi),
\intertext{where we have put}
\phi(x):= \frac{1}{2} J_\alpha(\sqrt{x}), \quad\psi(x):=  \frac{1}{2} \sqrt{x} J_{\alpha-1}(\sqrt{x}).
\end{gather*}
Hence the results of Theorem~\ref{thm:uv} simplify considerably: first, we obtain\footnote{Correcting an obvious typo, \eqref{eq:TWv} is precisely \cite[Eq.~(6)]{MR3513610}. As it was noted there, \eqref{eq:TWv} can also be found, though not explicitly, in \cite{MR1266485}. On the other hand, formula~\eqref{eq:Born_u} seems to be new.}
\begin{equation}\label{eq:TWv}
v(s) = \tr \big ((I-K)^{-1} \phi \otimes \phi \big)|_{L^2(0,s)} = \langle (I-K)^{-1}\phi, \phi\rangle_{L^2(0,s)};
\end{equation}
next, by observing
\[
K' (I-K)^{-1} K' = \langle (I-K)^{-1}\phi, \phi\rangle_{L^2(0,s)} \cdot \phi \otimes \phi = v(s)\cdot K',
\]
we get, because of symmetry and linearity,
\begin{equation}\label{eq:Born_u}
u(s) =  (1-\alpha)v(s) + \langle (I-K)^{-1}\phi, \psi\rangle_{L^2(0,s)} + v(s)^2.
\end{equation}
\end{subequations}
Both formulae for the auxiliary functions $u$ and $v$ can now be easily implemented in the author's Matlab toolbox for Fredholm determinants (which provides also commands to evaluate traces and inner products of general operator terms including resolvents; cf. \cite{MR2895091,MR2600548} and \cite[Appendix]{MR3647807}). Since all the numerical evaluations come with an estimate of the (absolute) error there, the implied approximation errors in computing the generating function $f_l(r)$ and its auxiliaries $a_l(r)$ and~$b_l(r)$ can straightforwardly be assessed. 

\begin{remark}\label{rem:F2prime}
A result similar to Theorem~\ref{thm:uv} holds for smooth integral kernels $K(x,y)$, with sufficient decay at $\infty$, which induce integral operators $K$ on $L^2(s,\infty)$ for all $s \in \R$. Here we define
the derived kernel as
\[
K'(x,y):= K_x(x,y) + K_y(x,y)
\]
and get, if $g(s) =\det(I-K)|_{L^2(s,\infty)}>0$ for all $s\in\R$, the logarithmic derivative 
\[
\frac{d}{ds} \log g(s) = -\tr \big ((I-K)^{-1} K'\big) |_{L^2(s,\infty)}.
\]
The proof goes by considering $K_s(x,y)=K(s+x,s+y)$ and transforming $(s,\infty)$ to $(0,\infty)$ by a shift. As an example, the Tracy--Widom distribution $F_2(s)$ used in \eqref{eq:BDJ99} is known to be given in terms of the Airy kernel determinant \cite{MR1236195},
\begin{subequations}
\begin{equation}
F_2(s) = \det(I-K)|_{L^2(s,\infty)}, \qquad K(x,y) = \frac{\Ai(x)\Ai'(y)-\Ai'(x)\Ai(y)}{x-y}.
\end{equation}
Here we have $K'(x,y) = -\Ai(x)\Ai(y)$, i.e., $K'=-\Ai\otimes \Ai$, and thus
\begin{equation}\label{eq:F2prime}
F_2'(s) = -F_2(s)\cdot \tr \big ((I-K)^{-1} K' \big)|_{L^2(s,\infty)}
= F_2(s)\cdot  \langle (I-K)^{-1}\Ai,\Ai\,\rangle_{L^2(s,\infty)}.
\end{equation}
The last formula was used for the calculations shown in Table~\ref{tab:3}. In the same manner we get
\begin{align}
F_2''(s) &= 2 F_2(s) \cdot \langle (I-K)^{-1} \Ai, \Ai'\,\rangle_{L^2(s,\infty)},\label{eq:F2doubleprime}\\[1mm]
F_2'''(s) &= 2 F_2(s) \cdot \left(\langle (I-K)^{-1} \Ai', \Ai'\,\rangle_{L^2(s,\infty)} + \langle (I-K)^{-1} \Ai, \Ai''\,\rangle_{L^2(s,\infty)}\right),
\end{align}
\end{subequations}
as well as similar formulae for higher order derivatives of $F_2$.\footnote{Though the inner products in \eqref{eq:F2prime} and \eqref{eq:F2doubleprime} appear in the work of Tracy and Widom \cite[Eq.~(1.3)]{MR1266485}, the formula for $F_2''(s)$ is not given there.}
\end{remark}

\subsection{Implementation details} First, by uniqueness, solving $a_l(r)=n$ for $r=r_{l,n}$ 
can easily be accomplished by an iterative solver. In view of the left panel in Fig.~\ref{fig:CLT} and the expansion \eqref{eq:rstirlingasymp} we take as initial guess 
\[
r_0 := \max(n, (n/l)^2 + n/2).
\] 
Second, the numerical evaluation of the Stirling-type formula \eqref{eq:stirling} for larger values of $n$ requires to avoid severe overflow of intermediate terms. Based on the representations in \eqref{eq:aux_in_rmt}, and by rearranging terms, we can write \eqref{eq:stirling}  equivalently as follows:\footnote{We have to stabilize the numerical evaluation of the expression $h- \log(1+h)$ for small $h:=v_l/n\approx 0$. This is done, first, by using {\tt h-log1p(h)} and, second, by switching to a suitable Taylor expansion for very small $h$.} 
\begin{equation}\label{eq:stirlingright}
\prob(L_n \leq l) = \tau_n \cdot g_l \cdot \frac{ \exp\left(n \left(\dfrac{v_l}{n}-\log\left( 1+ \dfrac{v_l}{n}\right)\right)\right)}{\sqrt{1 + \dfrac{v_l-u_l}{n}}}(1+ o(1)) \qquad (n\to\infty),
\end{equation}
where $g_l$, $v_l$ and $u_l$ are evaluated  at $s=4r_{l,n}$ and there is
\[
\tau_n := \frac{n!}{\sqrt{2\pi n}} \left(\frac{e}{n}\right)^n = 1+\frac{n^{-1}}{12}+\frac{n^{-2}}{288}-\frac{139n^{-3}}{51840}-\frac{571n^{-4}}{2488320}+\frac{163879n^{-5}}{209018880}+O(n^{-6}).
\]
In IEEE hardware arithmetic we take the definition of $\tau_n$ until $n=100$ and only switch to the shown Stirling expansion  for larger $n$---thus seamlessly providing full accuracy.\footnote{In fact, nothing of substance would change if we just replaced $\tau_n$ by $1$ since the thus committed error would be in the same ballpark as the one of the Stirling-type formula \eqref{eq:stirling} itself. We did not bother to do so, though.} This allows us to approximate the PDF $\prob(L_n = l)$ near its mode for up to $n = 10^{12}$ and larger. For accurate tails, such as in Table~\ref{tab:2}, we have to resort to higher precision arithmetic, though.

\section{First and Second Finite Size Corrections to the Random Matrix Limit}\label{sect:finite}

\subsection{The CDF of the distribution of \boldmath$L_n$\unboldmath}\label{sect:CDFexpansion}

Based on data from Monte-Carlo simulations, Forrester and Mays \cite{arxiv.2205.05257} have recently initiated the study of finite size corrections to the random matrix limit \eqref{eq:BDJ99}, which is
\begin{equation}\label{eq:tl}
\prob( L_n \leq l) = F_2(t_l) + o(1),\qquad t_l:=\frac{l-2\sqrt{n}}{n^{1/6}},
\end{equation}
as $n\to\infty$, uniformly in $l\in \N$. We will refine their study by using the much more accurate and efficient Stirling-type formula \eqref{eq:stirling} instead. Looking at the error 
\[
\delta_0(n) := \max_{l\in\{1,\ldots,n\}}\big|\prob( L_n \leq l) - F_2(t_l)\big|
\]
for $n$ up to $1000$ (see the red crosses in the left panel of Fig.~\ref{fig:error} in a double logarithmic scaling) suggest that $\delta_0(n) \approx c_1 n^{-1/3} + c_2 n^{-2/3}$ and yields the conjecture
\begin{equation}\label{eq:conj1}
\prob( L_n \leq l) = F_2(t_l)  + n^{-1/3} F_{2,1}(t_l) + O(n^{-2/3})
\end{equation}
for some function $F_{2,1}(t)$. Numerically, the conjecture has been convincingly checked against the data obtained by the Stirling-type formula \eqref{eq:stirling} for $n=10^6$, $n=10^8$, and $n=10^{10}$; see the left panel of Fig.~\ref{fig:1st}. We have fitted a polynomial $\tilde F_{2,1}(t)$ of degree 64 to the $836$ data points obtained for $n=10^{10}$ in the interval $-8 \leq t \leq 10$, thus approximating the putative  function $F_{2,1}(t)$ there. 

\begin{remark}\label{rem:F21error} The error of approximating $F_{2,1}(t)$ by this procedure can be estimated as follows.
Extrapolating the errors displayed in the left panel of Fig.~\ref{fig:error} shows that the Stirling-type formula induces a perturbation of size $\approx (0.031 n^{-2/3} + 0.058 n^{-1}) n^{1/3}|_{n=10^{10}} = 1.4\cdot 10^{-5}$. On the other hand, the finite size effect of the next order term $n^{-2/3} F_{2,2}(t)$, displayed in the left panel of Fig.~\ref{fig:2nd}, induces a perturbation of size $\approx 0.25 n^{-1/3}|_{n=10^{10}} = 1.2\cdot 10^{-4}$. Thus, altogether $\tilde F_{2,1}$ approximates $F_{2,1}$ up to an error\footnote{(added in proof) In fact, the conjectured analytic form \eqref{eq:F21} of $F_{2,1}(t)$  gives $\|\tilde F_{2,1} - F_{2,1}\|_\infty \approx 1.1814 \cdot 10^{-4}$.\label{foot:F21accuracy}} of the order $10^{-4}$.
\end{remark}
\begin{figure}[tbp]
\includegraphics[width=0.475\textwidth]{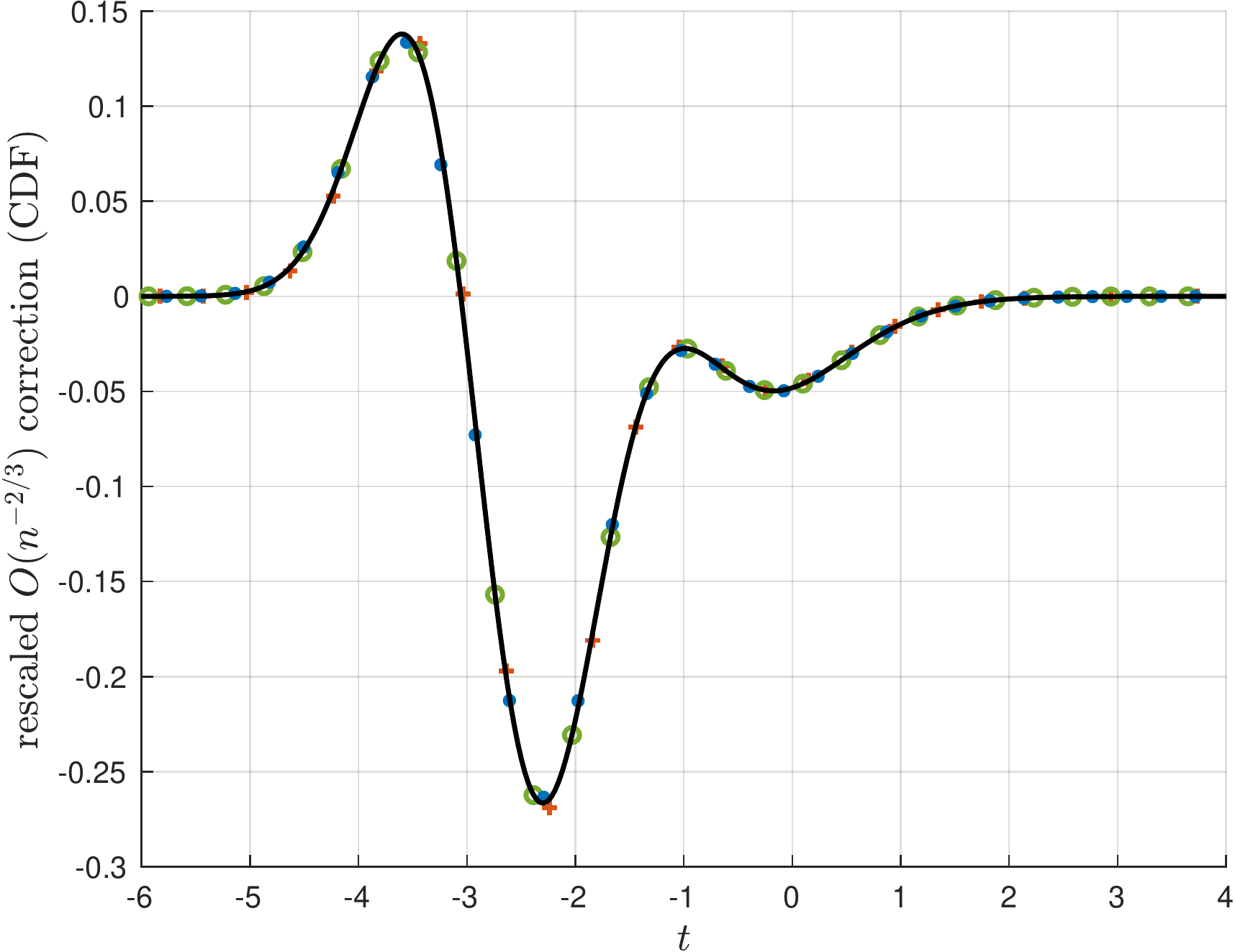}\hfil
\includegraphics[width=0.475\textwidth]{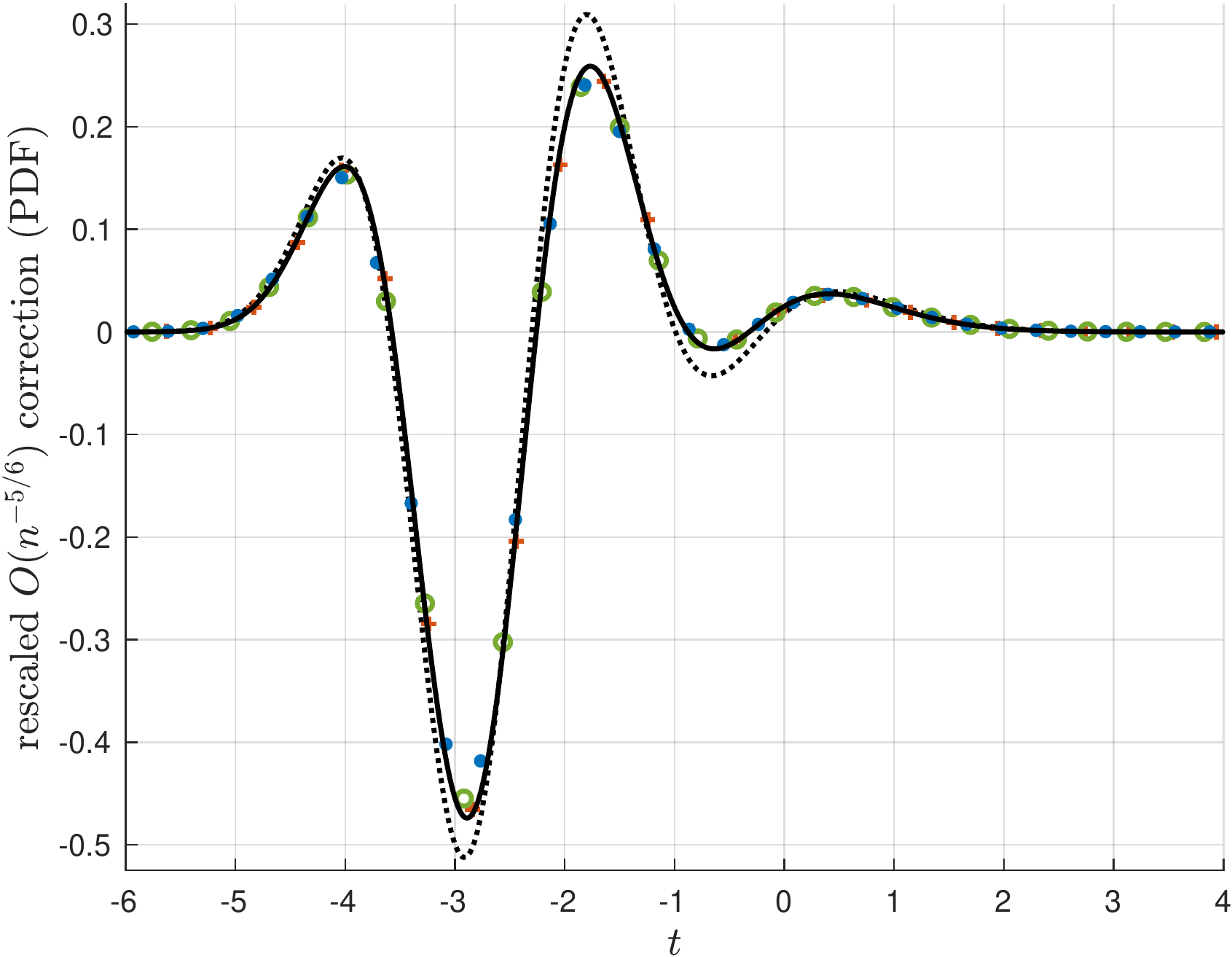}
\caption{{\footnotesize Rescaled differences between the distributions of $L_n$ and their expansions truncated after the first finite size correction term---see \eqref{eq:CDFexpansion} for the CDF resp. \eqref{eq:PDFexpansion} for the PDF; the data points have been calculated with the polynomial $\tilde F_{2,1}(t)$ from Fig.~\ref{fig:1st} for the exact values of the distribution for $n=250$ (red $+$), $n=500$ (green $\circ$), $n=1000$ (blue~$\bullet$). Left: CDF errors rescaled by $n^{2/3}$, horizontal axis is $t=(l-2\sqrt{n})/n^{1/6}$. The solid line is a polynomial $\tilde F_{2,2}(t)$ of degree $48$ fitted to all the $152$ data points
with $-7.5 \leq t \leq 9.5$; it approximates $F_{2,2}(t)$ in that interval. Right: PDF errors rescaled by $n^{5/6}$, horizontal axis is $t=(l-\frac12-2\sqrt{n})/n^{1/6}$. The solid line displays $\tilde F_{2,2}'(t)+\tilde F_{2,1}'''(t)/24 + F_2^{(5)}(t)/1920$ as an approximation of $F_{2,2}'(t)+ F_{2,1}'''(t)/24 + F_2^{(5)}(t)/1920$, with the polynomial $\tilde F_{2,2}(t)$ taken from the left panel and $\tilde F_{2,1}(t)$ as in Fig.~\ref{fig:1st}. The dotted line displays the term $\tilde F_{2,2}'(t)$ only.}}
\label{fig:2nd}
\end{figure}
If we iterate this approach yet another step, by looking at the error
\[
\delta_1(n) := \max_{l\in\{1,\ldots,n\}}\big|\prob( L_n \leq l) - F_2(t_l) - n^{-1/3} \tilde F_{2,1}(t_l)\big|
\]
for $n$ up to $1000$, then a double logarithmic plot (the green circles in the left panel of Fig.~\ref{fig:error}) suggests that $\delta_1(n) \approx c_1 n^{-2/3} + c_2 n^{-1}$. As stated in the introduction, this yields the refinement~\eqref{eq:CDFexpansion} of conjecture \eqref{eq:conj1}, namely that there is further a function $F_{2,2}(t)$ such that
\begin{equation}\label{eq:CDFexpansion2}
\prob(L_n\leq l) = F_2(t_l) + n^{-1/3} F_{2,1}(t_l) + n^{-2/3} F_{2,2}(t_l) + O(n^{-1}) \qquad (n\to\infty),
\end{equation}
uniformly in $l\in\N$.

\begin{remark}\label{rem:F22error}
To validate conjecture \eqref{eq:CDFexpansion2} against numerical data, obtained by replacing $F_{2,1}(t)$ by the approximation $\tilde F_{2,1}(t)$, we have to be careful with an effective choice of $n$, though. On the one hand, as a perturbation of $F_{2,2}(t)$ the error of about $10^{-4}$ in $\tilde F_{2,1}$, as estimated in Remark~\ref{rem:F21error}, would get amplified by $n^{1/3}$. On the other hand, an extrapolation of the errors displayed in the left panel of Fig.~\ref{fig:error} shows that the Stirling-type formula would induce an additional perturbation of size $\approx (0.031 n^{-2/3} + 0.058 n^{-1}) n^{2/3}$. The sweet spot of both perturbations combined is at $n\approx 1.4\cdot 10^4$ with a minimum error of about $3.6\cdot 10^{-2}$.
Thus, we better stay with the tabulated exact values of the distribution of $L_n$ up to $n=1000$, which restricts the size of the perturbation to just less than the order of $n^{1/3} 10^{-4}|_{n=1000} = 10^{-3}$.
\end{remark}

Thus, staying with the tabulated values of the distribution of $L_n$ for $n=250$, $n=500$, and $n=1000$ we get a convincing picture; see the left panel of  Fig.~\ref{fig:2nd}. We have fitted a polynomial $\tilde F_{2,2}(t)$ of degree $48$ to all of the $152$ data points in the interval $-7.5 \leq t \leq 9.5$, approximating the putative function $F_{2,2}(t)$ there. 

\subsection{The PDF of the distribution of \boldmath$L_n$\unboldmath}\label{sect:PDFexpansion}

If we apply the central differencing formula, for smooth functions $F(x)$ and increments $h\to 0$, that is to say
\[
F(x+h)-F(x) = h F'(x+h/2) +\frac{h^3}{24} F'''(x+h/2) + \frac{h^5}{1920} F^{(5)} (x+ h/2) + O(h^7),
\]
with the increment $h=n^{-1/6}$ to the conjectured expansion \eqref{eq:CDFexpansion2}, now written in the form
\begin{multline*}
\prob(L_n=l) = \prob(L_n\leq l)  - \prob(L_n\leq l-1)\\*[2mm] = \big(F_2(t_l)-F_2(t_{l-1}) \big) + n^{-1/3} \big(F_{2,1}(t_l)-F_{2,1}(t_{l-1}) \big) + n^{-2/3} \big(F_{2,2}(t_l) - F_{2,2}(t_{l-1}) \big) + \cdots,
\end{multline*}
we get, assuming some uniformity, an induced expansion of the PDF:
\begin{subequations}
\begin{multline}\label{eq:PDFexpansion}
\prob(L_n= l) = n^{-1/6} F_2'(\hat t_l) + n^{-1/2} \Big(F_{2,1}'(\hat t_l)+\frac{1}{24}F_2'''(\hat t_l)\Big)\\*[2mm]
 + n^{-5/6} \Big(F_{2,2}'(\hat t_l) + \frac{1}{24}F_{2,1}'''(\hat t_l)+\frac{1}{1920}F_2^{(5)}(\hat t_l)\Big) + O(n^{-7/6})\qquad (n\to\infty)
\end{multline}
uniformly in $l\in\N$, where we have briefly written 
\begin{equation}\label{eq:hattl}
\hat t_l := \frac{l-\frac{1}{2}-2\sqrt{n}}{n^{1/6}}.
\end{equation}
\end{subequations}
Note the shift by $1/2$ in the numerator of $\hat t_l$ as compared to $t_l$ (defined in \eqref{eq:tl}). There is compelling numerical evidence for the expansion \eqref {eq:PDFexpansion}; see the right panels of Figs.~\ref{fig:1st} and \ref{fig:2nd}.

\subsection{The expected value of \boldmath$L_n$\unboldmath}\label{sect:mean}

By a shift and rescale, the expected value of the discrete random variable $L_n$ can be written in the form 
\[
\Erw(L_n) = \sum_{l=1}^n l \cdot \prob(L_n=l) 
= 2\sqrt{n} + \frac{1}{2} + n^{1/6} \sum_{l=1}^n \hat t_l \cdot \prob(L_n=l).
\]
If we write \eqref{eq:PDFexpansion}, with obvious definitions of the functions $G_j(t)$, in the form
\begin{equation}\label{eq:probG}
\prob(L_n= l) = n^{-1/6} G_0(\hat t_l) + n^{-1/2} G_1(\hat t_l) + n^{-5/6} G_2(\hat t_l) + O(n^{-7/6})
\end{equation}
we get the induced expansion
\begin{align*}
\Erw(L_n) &= 2\sqrt{n} + n^{1/6} \mu_0^{(n)} + \frac{1}{2} + n^{-1/6} \mu_1^{(n)} + n^{-1/2} \mu_2^{(n)} + \cdots \\*[1mm]
\mu_j^{(n)} &:= n^{-1/6} \sum_{l=1}^n \hat t_l \, G_j(\hat t_l).
\end{align*}
Now, if we assume (a) that  the decay $G_j(t) \to 0$ (and likewise for all the derivatives) is exponentially fast as $t\to\pm\infty$ (see Figs.~\ref{fig:1st} and \ref{fig:2nd}) and (b) that  the $G_j$ can be extended analytically to a strip containing the real axis, we obtain
\begin{equation}\label{eq:trap}
\mu_j^{(n)} \doteq n^{-1/6} \sum_{l=-\infty}^\infty \hat t_l \, G_j(\hat t_l) \doteq \int_{-\infty}^\infty t\, G_j(t)\,dt =:\mu_j,
\end{equation}
where “$\doteq$” denotes equality up to terms that are exponentially small for large $n$. Here, in the first step
the series was obtained by adding, under assumption (a), the exponentially small tail, and in the next step  we have identified the series as the trapezoidal rule with step-size $h = n^{-1/6}$---a quadrature rule known to converge, under assumption~(b), exponentially fast to the integral, cf. \cite{MR3245858}.

\begin{table}[tbp]
\caption{Exponentially fast convergence of first ($k=1$) and second ($k=2$) moments,\\*[2mm]
\hspace*{2.45cm}$\mu_{0,k}^{(n)} := n^{-1/6} \sum_{l=1}^n \hat t_l^k F_2'(\hat t_l) \;\to\; 
\mu_{0,k}^{(\infty)} := \int_{-\infty}^\infty t^k F_2'(t)\,dt.$\\*[-4mm]
}
\label{tab:3}
{\footnotesize
\begin{tabular}{cccc}
\hline\noalign{\smallskip}
$n$ & $\mu^{(n)}_{0,1}$ & $\mu^{(n)}_{0,2}$ &\;\;\,$\mu^{(n)}_{0,2} - \big(\mu^{(n)}_{0,1}\big)^2$ \\
\noalign{\smallskip}\hline\noalign{\smallskip}
                         $6$        & $-1.73195\,96234$ & $3.76769\,77551$ & $0.76801\,36177$ \\      
                        $12$        & $-1.77034\,42726$ & $3.94627\,23262$ & $0.81215\,34824$ \\         
                        $24$        & $-1.77108\,66107$ & $3.94994\,20793$ & $0.81319\,42965$ \\        
                        $48$        & $-1.77108\,68074$ & $3.94994\,32722$ & $0.81319\,47928$ \\*[-1mm]
                        $\vdots$    & $\vdots$ & $\vdots$ & $\vdots$                          \\
                        $\infty$    & $-1.77108\,68074$ & $3.94994\,32722$ & $0.81319\,47928$ \\    
\noalign{\smallskip}\hline
\end{tabular}}
\end{table}%

\begin{remark} For  the function $G_0(t) = F_2'(t)$, i.e., the density of the Tracy--Widom distribution, the asserted exponentially fast convergence \eqref{eq:trap} can be checked against numerical data which were obtained by applying the highly accurate numerical methods described in \cite{MR2895091} to the representation \eqref{eq:F2prime} of $F_2'(t)$; see Table~\ref{tab:3}.
\end{remark}

We have thus derived from \eqref{eq:CDFexpansion2}---based on the assumptions of uniformity, exponential decay, and analytic continuation of the functions $G_j$ and their derivatives---the following expansion which adds three more terms to the expansion given in \cite[Thm.~1.2]{MR1682248}:
\begin{equation}\label{eq:mean2}
\Erw(L_n) = 2\sqrt{n} +  \mu_0 n^{1/6} + \frac{1}{2} + \mu_1n^{-1/6}  + \mu_2n^{-1/2}  + O(n^{-5/6});
\end{equation}
where we have, with the numerical value of $\mu_0$ taken from \cite[Table~10]{MR2895091}, cf. also Table~\ref{tab:3}, 
\begin{align*}
 \mu_0 &= \int_{-\infty}^\infty t \, F_2'(t)\,dt = -1.77108\,68074\,\cdots,\\*[1.75mm] 
 \mu_1 &= \int_{-\infty}^\infty t\Big(F_{2,1}'(t)+\frac{1}{24}F_2'''(t)\Big)\,dt = \int_{-\infty}^\infty t \, F_{2,1}'(t) \,dt,\\*[1.75mm]
 \mu_2&= \int_{-\infty}^\infty t\Big(F_{2,2}'(t) + \frac{1}{24}F_{2,1}'''(t)+\frac{1}{1920}F_2^{(5)}(t)\Big)\,dt = \int_{-\infty}^\infty t \, F_{2,2}'(t) \,dt.
 \end{align*}
We note that the higher derivatives do not contribute to the integral values, as can be shown using integration by parts and the assumed exponential decay to zero. Based on the polynomial approximations displayed in Figs.~\ref{fig:1st} and \ref{fig:2nd} we get the  numerical estimates---comparing, in addition to $n=10^{10}$, with the analogous results for $n=10^9$ and $n=10^{11}$:
\begin{equation}\label{eq:integrals}
\mu_1 \approx \int_{-8}^{10} t \tilde F'_{2,1}(t)\,dt \approx 0.0659, \qquad \mu_2 \approx \int_{-7.5}^{9.5} t \tilde F'_{2,2}(t)\,dt \approx 0.25.
\end{equation}
Further evidence for the validity of the expansion \eqref{eq:mean2} comes from looking at a least squares fit of the form\footnote{Forrester and Mays \cite[Sect.~4.4]{arxiv.2205.05257} discussed a least squares fit of the
form
\[
\Erw(L_n) \approx 2\sqrt{n} -1.77108\,68074 \cdot n^{1/6} + \hat{c}_1 + \hat{c}_2 n^{-\alpha},
\]
where they let $\alpha=1/6$ compete with $\alpha = 1/3$. Using exact values of $\Erw(L_n)$ for $n$ from~$10$ up to $700$, they identified the exponent $\alpha=1/3$ to provide the better fit, with values $\hat{c}_1 = 0.507$ and $\hat{c}_2 = 0.222$. They give reasons (different from ours) to expect $\hat{c}_1=0.507$ to correspond to an exact constant term $1/2$ in the expansion
of $\Erw(L_n)$. However, we can fully explain their result by just taking the least squares fit of their ansatz to the relevant terms of the expansion~\eqref{eq:mean2}, that is by fitting the simplified model
\[
0.5 + 0.0658 n^{-1/6} + 0.261 n^{-1/2}-0.119 n^{-5/6} \approx \tilde{c}_1 + \tilde{c}_2 n^{-\alpha}
\]
for $n=10,\ldots,700$. Unsurprisingly, $\alpha=1/3$ is the better choice over $\alpha=1/6$ here; and we get $\tilde c_1 =  0.506$ and $\tilde c_2 = 0.222$, reproducing the values reported by them. 
} (where the upper bound $k=9$ has been chosen as to maximize the number of matching digits for the two data sets below)
\[
\Erw(L_n) \approx 2\sqrt{n} + \frac{1}{2} + \sum_{k=0}^9 c_k n^{(1-2k)/6},
\]
with $\Erw(L_n)$ obtained from the tabulated values of $\prob(L_n=l)$ up to $n=1000$. If we do so in extended precision for two different data sets, first for~$n$ from~$500$ upwards and next for $n$ from $600$ upwards, we obtain as digits that are matching in both cases
\[
c_0 = -1.77108\,68074\cdots, \; c_1 = 0.06583\,238\cdots, \; c_2 = 0.26122\,27\cdots, \; c_3 = -0.11938\,4\cdots.
\]
Here the value of $c_0$ is in perfect agreement with the known value of $\mu_0$ and $c_1$, $c_2$ are consistent with the inaccuracies of the estimates in \eqref{eq:integrals}, cf. Remarks~\ref{rem:F21error}/\ref{rem:F22error}. Hence, the most accurate values that we can offer for $\mu_1$ and $\mu_2$ are those displayed in \eqref{eq:mean}.

\subsection{The variance of \boldmath$L_n$\unboldmath}\label{sect:var}

By a shift and rescale, the variance of $L_n$ can we written as
\[
\Var(L_n) = \sum_{l=1}^n l^2 \cdot \prob(L_n=l)  - \Erw(L_n)^2 
= n^{1/3} \sum_{l=1}^n \hat t_l^2
\cdot \prob(L_n=l) - \left(\Erw(L_n)- 2\sqrt{n} -\frac12\right)^2.
\]
By inserting the expansions \eqref{eq:probG} and \eqref{eq:mean2}, and by arguing as in \eqref{eq:trap}, we get the following expansion which adds two more terms to the leading order found in \cite[Thm.~1.2]{MR1682248}:
\begin{equation}\label{eq:var}
\Var(L_n) = \nu_0 n^{1/3} + \nu_1 + \nu_2 n^{-1/3} + O(n^{-2/3});
\end{equation}
where we have, with the numerical value of  $\nu_0$ taken from \cite[Table~10]{MR2895091}, cf. also Table~\ref{tab:3},\footnote{Note that $\frac{1}{24}\int_{-\infty}^{\infty} t^2 F'''_2(t)\,dt = \frac{2}{24} \int_{-\infty}^{\infty} F'_2(t)\,dt= \frac{1}{12}$.}
\begin{gather*}
\nu_0 = \int_{-\infty}^\infty t^2 F_2'(t)\,dt - \mu_0^2 = 0.81319\,47928\,\cdots,\\*[1mm]
\nu_1 = \int_{-\infty}^\infty t^2 F_{2,1}'(t)\,dt + \frac{1}{12} -2\mu_0\mu_1 ,\qquad 
\nu_2 = \int_{-\infty}^\infty t^2 F_{2,2}'(t)\,dt -\mu_1^2 - 2\mu_0\mu_2.
\end{gather*}
By using the polynomial approximations $\tilde F_{2,1}$ and $\tilde F_{2,2}$ displayed in Figs.~\ref{fig:1st} and \ref{fig:2nd}, as well as the numerical values of $\mu_0$, $\mu_1$, $\mu_2$ from \eqref{eq:mean}, we get the estimates $\nu_1 \approx -1.2070$ and $\nu_2 \approx 0.57$. 

Once again, further evidence and increased numerical accuracy comes from a least squares fit of the form\footnote{See \cite[Eq.~(4.19)]{arxiv.2205.05257} for a less accurate fit of the form $\Var(L_n) \approx 0.81319\,47928\cdot  n^{1/3} + \tilde c_1 + \tilde c_2 n^{-1/3}$.}
(where the upper bound $k=8$ has been chosen as to maximize the number of matching digits for the two data sets below)
\[
\Var(L_n) \approx \sum_{k=0}^8 c_k n^{(1-k)/3},
\]
with $\Var(L_n)$ obtained from the tabulated values of $\prob(L_n=l)$ up to $n=1000$. If we do so in extended precision for two different data sets, first for~$n$ from~$500$ upwards and then for $n$ from $600$ upwards, we obtain as digits that are matching in both cases
\[
c_0 = 0.81319\,47928\cdots,\; c_1 = -1.20720\,507\cdots,\; c_2 = 0.56715\,6\cdots,\; c_3 = 0.01669\cdots.
\]
Here the value of $c_0$ is in perfect agreement with the known value of $\nu_0$ and the values for $c_1$, $c_2$ are consistent with the inaccuracies of the estimates for $\nu_1$ and $\nu_2$ shown above, cf. Remarks~\ref{rem:F21error}/\ref{rem:F22error}. Hence, the most accurate values that we can offer for the coefficients $\nu_1$ and $\nu_2$ are those displayed in \eqref{eq:var0}.

{\footnotesize\begin{remark}[added in proof]\label{rem:F21mean} The conjectured functional form \eqref{eq:F21} of $F_{2,1}(t)$  gives numerical values of the coefficients $\mu_1$ and $\nu_1$ that agree with the values displayed in \eqref{eq:mean} and \eqref{eq:var0} to all digits shown.
\end{remark}}

\setcounter{section}{0}
\renewcommand{\thesection}{\Alph{section}} 
\section{Appendix: The Multidimensional Laplace Method}

The classical one-dimensional method of Laplace can be generalized to provide the asymptotics, as $z\to\infty$,  of multidimensional integrals of the form
\[
\int_\Omega e^{-z S(x)} f(x)\,dx.
\]
Here, $\Omega\subset \R^n$ is a measurable set and $f, S : \Omega \to \R$ are subject to suitable assumptions. 
For instance, if we assume that $f$ and $S$ are sufficiently smooth and the phase function $S(x)$ takes a unique minimum at an interior point $x_*$ of $\Omega$ with $\det S''(x_*)\neq 0$ then the standard result---going back to Hsu \cite[Lemma 1]{Hsu48}---states that for each $0<\delta\leq \frac\pi2$ as $z\to\infty$
\begin{equation}\label{eq:Hsu}
\int_\Omega e^{-z S(x)} f(x)\,dx = \left(\frac{2\pi}{z}\right)^{n/2} \frac{e^{-z S(x_*)}}{\sqrt{ \det S''(x_*) }} \left(f(x_*) + O\big(z^{-1}\big)\right) \quad (|\!\arg z| \leq \tfrac{1}{2}\pi-\delta).
\end{equation}
\begin{remark}\label{rem:bigo}
As is customary in asymptotic analysis in the complex plane, cf. \cite[p.~7]{MR0435697}, we understand this asymptotic expansion (and similar expansions with $o$- or $O$-terms) to hold {\em uniformly} in $\{z\in \C: |z|\geq R_\delta, |\arg z| \leq \frac{\pi}{2}-\delta\}$ if $R_\delta>0$ has been chosen sufficiently large.
\end{remark}
If $f(x_*)=0$, however, formula \eqref{eq:Hsu} fails to yield the precise leading order term of the expansion. On the other hand, it is known that there holds, for sufficiently smooth $f$ and $S$, cf. \cite[Eq.~(1.27)]{Fedoryuk89}, a general asymptotic expansion of the form 
\begin{equation}\label{eq:fullasympt}
\int_\Omega  e^{-z S(x)} f(x)\,dx \sim  e^{-z S(x_*)}z^{-n/2}\sum_{k=0}^\infty c_k z^{-k} \qquad (z\to \infty,\,|\!\arg z| \leq \tfrac{1}{2}\pi-\delta).
\end{equation}
Still, it would be extremely awkward 
to determine the first non-zero coefficient from the standard proof\footnote{See, e.g.,
\cite[Eq.~(8.3.52)]{Bleistein86}, \cite[Thm.~IX.3]{Wong89}, \cite[Thm.~15.2.5]{Simon15} for real $z$ and \cite[Eq.~(1.26)]{Fedoryuk89} for complex $z$.} of \eqref{eq:fullasympt}, in which the $c_k$ depend on higher order derivatives of a nonlinear transform obtained from the Morse lemma, deforming $S(x)$ to a quadratic form.

Building on a different technique introduced by Fulks and Sather \cite{Fulks61}, Kirwin \cite{Kirwin10} succeeded in establishing formulae for the higher order coefficients in terms of asymptotic expansions of $f$ and $S$ into homogeneous functions at $x_*$.
However, these authors consider only the case of {\em real} $z\to\infty$, whereas we need uniformity, for arbitrary small $\delta>0$, as $z\to \infty$ with $|\arg z| \leq \frac{\pi}{2}-\delta$. Since the leading order term of their expansions suffices for the purposes of Sect.~\ref{sect:hayman}, we will give  a much simplified version of their proof in this Appendix, explicitly tracing constants to establish the required uniformity. 

\subsection*{Notation and assumptions} By a simple transformation (see the proof of Thm.~\ref{thm:main}) we can restrict ourselves to
\[
x_* = 0, \qquad S(x_*)=0.
\]
Writing $z = \sigma + i\tau$ we have
\begin{equation}\label{eq:sigmadelta}
0< \sigma \leq |z| \leq  \sigma\csc \delta\qquad (|\!\arg z| \leq \tfrac{1}{2}\pi-\delta).
\end{equation}
Let $\Omega \subset \R^n$ we a measurable set with $0$ as an interior point, that is, there is $\epsilon_0 >0$  such that $B_\epsilon(0)\subset \Omega$ for all open balls of radius $0< \epsilon\leq \epsilon_0$ centered at $0$. By denoting $|\cdot|$ the Euclidean norm on $\R^n$, we write $\R^n\setminus\{0\} \ni x = \rho \cdot \xi$ in spherical coordinates with $\rho=|x|$ and $\xi=x/\rho \in S^{n-1}$. We assume that $f, S: \Omega \to \R$ are measurable functions subject to the following conditions:
\begin{enumerate}
\item $S(x)$ is positively bounded away from zero on $\Omega\setminus B_\epsilon(0)$ for each $\epsilon>0$;\\*[-4mm]
\item there is a $\nu>0$ and a positive continuous function $S_0:S^{n-1}\to\R$ with
\[
S(x) = \rho^\nu S_0(\xi) + o(\rho^\nu)\qquad (\rho\to 0),
\]
uniformly in $\xi\in S^{n-1}$;
\item there is a $\lambda > 0$ and a bounded measurable function $f_0:S^{n-1}\to \R$ with
\[
f(x) = \rho^{\lambda-n} f_0(\xi) + o(\rho^{\lambda-n})\qquad (\rho\to 0),
\]
uniformly in $\xi\in S^{n-1}$;
\item integrability: there is $\sigma_0>0$ with $M = \int_\Omega e^{-\sigma_0 S(x)}|f(x)| \,dx <\infty$.
\end{enumerate}
\medskip
We extend $f_0$ and $S_0$ to all of $\R^n$ by homogeneity, that is, by
\[
f_0(x) = \rho^{\lambda-n} f_0(\xi), \quad S_0(x) = \rho^{\nu} S_0(\xi) \qquad (x\neq 0)
\]
and, for definiteness, $f_0(0)=S_0(0)=0$. Finally, for purposes of reference we recall the following well known integral evaluation,
\begin{equation}\label{eq:gamma}
\int_0^\infty e^{-z \rho^\nu} \rho^{\lambda - 1}\,d\rho = \frac{\Gamma(\lambda/\nu)}{\nu} z^{-\lambda/\nu}\qquad (\sigma = \Re z > 0).
\end{equation}

The leading order result of \cite[p.~186]{Fulks61}\footnote{Note that there is a typo in \cite[p.~186]{Fulks61}: the constant has to be $\Gamma(\lambda/\nu)/\nu$ rather than $\Gamma((\lambda+1)/\nu)/\lambda$.} and \cite[Thm.~1.1]{Kirwin10} is now as follows:

\begin{theorem}\label{thm:Fulks} Under conditions {\rm (1)--(4)} there holds for each $0< \delta\leq \frac\pi2$ as $z\to\infty$
\begin{equation}\label{eq:Fulks}
\int_\Omega e^{-z S(x)} f(x)\,dx = z^{-\lambda/\nu} \int_{\R^n} e^{- S_0(x)} f_0(x)\,dx + o(z^{-\lambda/\nu})\qquad (|\!\arg z| \leq \tfrac{1}{2}\pi-\delta).
\end{equation}
Denoting the surface measure on $S^{n-1}$ by $\omega$, the integral on the right evaluates to 
\[
\int_{\R^n} e^{- S_0(x)} f_0(x)\,dx = \frac{\Gamma(\lambda/\nu)}{\nu}  \int_{S^{n-1}} \frac{f_0(\xi)}{S_0(\xi)^{\lambda/\nu}}\,d\omega(\xi).
\]
\end{theorem}

We fix $\delta>0$ and break the proof of this theorem, based on the idea of “trading tails” (a notion popularized for Laplace's method in \cite[p.~466]{MR1397498}), into some preparatory steps.

\begin{lemma}\label{lem:bulk} For $\sigma\geq \sigma_0$ and $0<\epsilon \leq \epsilon_0$ there is a constant $\alpha>0$ such that
\[
\int_{\Omega\setminus B_\epsilon(0)} e^{-zS(x)} f(x)\,dx = O\left(e^{-(\sigma-\sigma_0)\alpha \epsilon^\nu}\right)
\]
where the implied constant does not depend on $z$ and $\epsilon$.
\end{lemma}
\begin{proof} By Conditions (1) and (2) there is a constant $\alpha>0$ such that
\[
S(x) \geq \alpha \epsilon^\nu\qquad (x\in\Omega\setminus B_\epsilon(0)).
\]
Hence, Condition (4) yields straightforwardly
\[
\left|\int_{\Omega\setminus B_\epsilon(0)} e^{-zS(x)} f(x)\,dx\right| \leq e^{-(\sigma - \sigma_0)\alpha \epsilon^\nu} \int_\Omega e^{-\sigma_0 S(x)} |f(x)|\,dx = O\left(e^{-(\sigma-\sigma_0) \alpha\epsilon^\nu}\right),
\]
the implied constant being just $M$.
\end{proof}

\begin{lemma}\label{lem:fdiff} There is $\eta_\epsilon \to 0$ for $\epsilon\to 0$ such that for $\sigma\geq \sigma_0$ and $0<\epsilon \leq \epsilon_0$ 
\[
\int_{B_\epsilon(0)} e^{-zS(x)} f(x)\,dx = \int_{B_\epsilon(0)} e^{-zS(x)} f_0(x)\,dx  + \eta_\epsilon \cdot O\left(\sigma^{-\lambda/\nu}\right)
\]
where the implied constant does not depend on $z$ and $\epsilon$.
\end{lemma}
\begin{proof} Conditions (1) and (2) give the existence of a constant $\gamma>0$ such that
\[
S(x) \geq \gamma |x|^\nu\qquad (|x|\leq \epsilon_0)
\]
and Condition (3) gives $\eta_\epsilon\to 0$ for $\epsilon\to 0$ with
\[
|f(x)-f_0(x)| \leq \eta_\epsilon |x|^{\lambda-n}\qquad (|x|\leq \epsilon).
\] 
Hence we get 
\begin{multline*}
\left|\int_{B_\epsilon(0)} e^{-zS(x)} (f(x)-f_0(x))\,dx \right| \leq \eta_\epsilon \int_{\R^n} e^{-\sigma \gamma |x|^\nu} |x|^{\lambda-n}\,dx \\
=  \eta_\epsilon  \omega_{n-1} \int_0^\infty e^{-\sigma \gamma \rho^\nu} \rho^{\lambda-1}\,d\rho = 
\eta_\epsilon  \omega_{n-1} \nu^{-1} \Gamma(\lambda/\nu) (\sigma \gamma)^{-\lambda/\nu}
\end{multline*}
where the integral over $\R^n$ was evaluated by spherical symmetry and the resulting gamma integral by Eq.~\eqref{eq:gamma}; $\omega_{n-1}$ denotes the surface area of the sphere $S^{n-1}$. 
\end{proof}

\begin{lemma}\label{lem:sdiff} There is $\eta'_\epsilon \to 0$ for $\epsilon\to 0$ such that, 
for $0<\epsilon \leq \epsilon_0$ with $\epsilon_0$ sufficiently small and $z\in \C$ with $|\!\arg z|\leq \frac\pi2 - \delta$, $|z|\geq \sigma_0\cdot\csc\delta$,
\[
\int_{B_\epsilon(0)} e^{-zS(x)} f_0(x)\,dx = \int_{B_\epsilon(0)} e^{-zS_0(x)} f_0(x)\,dx  + \eta'_\epsilon \cdot O\left(z^{-\lambda/\nu}\right)
\]
where the implied constant does not depend on $z$ and $\epsilon$.
\end{lemma}
\begin{proof} Condition (2) gives $\eta''_\epsilon\to 0$ for $\epsilon\to 0$ and $\alpha'>0$ with
\begin{equation}\label{eq:alphaprime}
|S(x)-S_0(x)| \leq \eta''_\epsilon |x|^\nu\quad (|x|\leq \epsilon) \quad\text{and}\quad S_0(x) \geq \alpha' |x|^\nu\quad (x\in\R^n),
\end{equation}
whereas Condition (3) yields a constant $\gamma'>0$ such that
\[
|f_0(x)| \leq \gamma' |x|^{\lambda-n} \qquad (x\in \R^n).
\]
Because of $|e^w-1| \leq e^{|w|}-1$ for $w\in\C$ and by \eqref{eq:sigmadelta} we obtain for $x\in B_\epsilon(0)$
\begin{multline*}
\left|e^{-zS(x)} -e^{-zS_0(x)}\right| \leq e^{-\sigma S_0(x)} \left( e^{\eta''_\epsilon |z| \cdot |x|^\nu} - 1 \right) \leq 
e^{-\sin(\delta) \alpha' |z| \cdot  |x|^\nu}  \left( e^{\eta''_\epsilon |z| \cdot |x|^\nu} - 1 \right) \\
\leq e^{-\sin(\delta)  (\alpha'-\eta''_\epsilon \csc \delta) |z|\cdot |x|^\nu} - e^{-\sin(\delta) \alpha' |z| \cdot |x|^\nu}
\end{multline*}
Thus, if $\epsilon_0$ is small enough to guarantee $\eta''_\epsilon \csc\delta < \alpha'$, we obtain by the same calculations as previously in the proof of  Lemma~\ref{lem:fdiff}
\begin{multline*}
\left|\int_{B_\epsilon(0)} \left(e^{-zS(x)}-e^{-zS_0(x)}\right) f_0(x)\,dx \right| \leq \\
\gamma' \omega_{n-1}\frac{\Gamma(\lambda/\nu)}{\nu} (\sin(\delta)|z|)^{-\lambda/\nu}\left(\frac{1}{(\alpha'-\eta''_\epsilon \csc\delta)^{\lambda/\nu}}-\frac{1}{(\alpha')^{\lambda/\nu}}\right) = \eta_\epsilon' \cdot O(z^{-\lambda/\nu})
\end{multline*}
if we define the term given by the large bracket to be $\eta_\epsilon'$.
\end{proof}

\begin{proof}[{\bf Proof of Thm.~\ref{thm:Fulks}}] By splitting the integral as follows and by applying Lemma~\ref{lem:bulk}--\ref{lem:sdiff} we get  for $0<\epsilon\leq \epsilon_0$ with $\epsilon_0$ sufficiently small and for $z\in\C$ with $|\!\arg z|\leq \frac\pi2 - \delta$, $|z|\geq \sigma_0\cdot\csc\delta$
\begin{multline*}
\int_\Omega e^{-zS(x)} f(x)\,dx = \int_{B_\epsilon(0)} e^{-z S_0(x)} f_0(x)\,dx + \int_{B_\epsilon(0)} (e^{-z S(x)}-e^{-z S_0(x)}) f_0(x)\,dx\\
+ \int_{B_\epsilon(0)} e^{-zS(x)} (f(x)-f_0(x))\,dx + \int_{\Omega\setminus B_\epsilon(0)} e^{-zS(x)} f(x)\,dx\\
= \int_{B_\epsilon(0)} e^{-zS_0(x)} f_0(x)\,dx + \eta_\epsilon' \cdot O(z^{-\lambda/\nu}) + \eta_\epsilon \cdot O(\sigma^{-\lambda/\nu}) +  O\left(e^{-(\sigma-\sigma_0)\alpha\epsilon^\nu}\right) 
\end{multline*}
where the implied constants do not depend on $z$ and $\epsilon$.
By \eqref{eq:alphaprime} and Lemma~\ref{lem:bulk} we get likewise
\[
\int_{\R^n} e^{-z S_0(x)} f_0(x)\,dx = \int_{B_\epsilon(0)} e^{-zS_0(x)} f_0(x)\,dx  +  O\left(e^{-(\sigma-\sigma_0)\alpha'\epsilon^\nu}\right).
\]
Coupling $\epsilon\propto |z|^{-1/2\nu}$ we thus get some expression $\eta(|z|^{-1})\to 0$ for $z\to\infty$ with
\[
\int_\Omega e^{-zS(x)} f(x)\,dx = \int_{\R^n} e^{-z S_0(x)} f_0(x)\,dx + \eta(|z|^{-1})\cdot O(z^{-\lambda/\nu})
\]
where the implied constant does not depend on $z$. Finally, by transforming to spherical coordinates and using \eqref{eq:gamma} for the inner integral once more, we calculate
\begin{multline*}
\int_{\R^n} e^{-z S_0(x)} f_0(x)\,dx = \int_{S^{n-1}}\int_0^\infty e^{- z S_0(\xi) \rho^\nu} f_0(\xi)\rho^{\lambda-1}\,d\rho \,d\omega(\xi) \\
= z^{-\lambda/\nu}\frac{\Gamma(\lambda/\nu)}{\nu} \int_{S^{n-1}} \frac{f_0(\xi)}{S_0(\xi)^{\lambda/\nu}}\,d\omega(\xi) 
\end{multline*}
which finishes the proof.
\end{proof}

\subsection*{Quadratic leading order term in the phase function}
It is straightforward from \eqref{eq:Fulks} to specialize Thm.~\ref{thm:Fulks} to the case of a quadratic leading order term $S_0$ in the asymptotic expansion of the phase function $S$.
 
\begin{corollary}\label{cor:Laplace} Under conditions {\rm (1)--(4)} with a quadratic $S_0(x) = x^T H x/2$ defined by a symmetric positiv definite matrix $H\in\R^{n\times n}$, there holds for each $0<\delta\leq \frac\pi2$ as $z\to\infty$
\[
\int_\Omega e^{-zS(x)} f(x)\,dx = \frac{(2\pi)^{n/2} z^{-\lambda/2}}{\sqrt{\det H}}(\E(f_0)+o(1))   \qquad (|\!\arg z| \leq \tfrac{1}{2}\pi-\delta)
\]
where $\E$ denotes expectation with respect to the multivariate normal distribution with covariance matrix $H^{-1}$, namely
\[
\E(f_0) := \frac{\sqrt{\det H}}{(2\pi)^{n/2}}\int_{\R^n} e^{-x^THx/2} f_0(x)\,dx.
\]
\end{corollary}

\begin{remark}\label{rem:Laplace} Corollary~\ref{cor:Laplace} is providing the precise leading order asymptotic of the integral only if the condition $\E(f_0) \neq 0$ is satisfied. In the case of a sufficiently smooth integrand the function $f_0$ is the first non-zero homogeneous polynomial appearing in the Taylor expansion of $f$ at zero.  
For symmetry reasons, $\E(f_0) \neq 0$ implies that $\deg f_0$ must be even. Thus, if also $S$ is sufficiently smooth,
a comparison with \eqref{eq:fullasympt} yields, if $\E(f_0)\neq 0$,
\begin{equation}\label{eq:laplacezero}
\int_\Omega e^{-zS(x)} f(x)\,dx = \frac{(2\pi)^{n/2} z^{-\frac{n+\deg f_0}{2}}}{\sqrt{\det H}}\E(f_0)\left(1+O(z^{-1})\right)
 \qquad (|\!\arg z| \leq \tfrac{1}{2}\pi-\delta).
\end{equation}
If $f(0)\neq 0$, this reproduces Hsu's formula \eqref{eq:Hsu} since then $f_0\equiv f(0)$ and hence $\E(f_0)=f(0)$.
\end{remark}

\subsection*{Acknowledgements} The author would like to thank the Isaac Newton Institute for Mathematical Sciences, Cambridge (UK), for support and hospitality during the 2019 program “Complex analysis: techniques, applications and computations (CAT)” where work on Sect.~\ref{sect:hayman} of this paper was undertaken. This work was supported by EPSRC grant no EP/R014604/1.

\bibliographystyle{spmpsci}
\bibliography{paper}

\end{document}